\documentclass[a4paper,11pt]{amsart}

\usepackage{amssymb,amsfonts,amsthm,amscd,stmaryrd}

\usepackage{ae}
\usepackage[utf8]{inputenc}
\numberwithin{figure}{section}

\usepackage{mathrsfs,a4wide}
\usepackage{esint}

\usepackage{color}

\usepackage{import}

\usepackage[leqno]{amsmath}

\numberwithin{figure}{section}

\newtheorem{theorem}{Theorem}[section]
\newtheorem{lemma}[theorem]{Lemma}
\newtheorem{proposition}[theorem]{Proposition}

\theoremstyle{definition}
\newtheorem{definition}[theorem]{Definition}

\newtheorem{remark}[theorem]{Remark}
\numberwithin{equation}{section}

\newcommand{\dloc}{a_{\rm stable}}

\newcommand{\wto}{\rightharpoonup}
\newcommand{\de}{\delta}
\newcommand{\R}{\mathbb{R}}
\newcommand{\N}{\mathbb{N}}
\newcommand{\Z}{\mathbb{Z}}
\newcommand{\Ha}{\mathcal{H}}

\newcommand{\beq}{\begin{equation}}
\newcommand{\eeq}{\end{equation}}

\newcommand{\eps}{\varepsilon}
\newcommand{\e}{\varepsilon}
\newcommand{\vphi}{\varphi}

\newcommand{\diver}{\operatorname{div}}
\newcommand{\Div}{\operatorname{div}}
\newcommand{\pa}{\partial}

\newcommand{\Om}{\Omega}
\newcommand{\medint}{-\kern -,375cm\int}
\newcommand{\medintinrigo}{-\kern -,315cm\int}
\newcommand{\C}{\mathbb{C}}
\newcommand{\Sb}{\mathbb{S}}

\renewcommand{\H}{\mathcal{H}}
\begin{document}

\title[Surface diffusion with elasticity in 2D]{The surface diffusion flow with elasticity in the plane}

\author{Nicola Fusco}

\author{Vesa Julin}

\author{Massimiliano Morini}


\begin{abstract} 
In this paper we prove  short-time existence of a smooth solution in the plane to the surface diffusion equation with an elastic term and without an additional curvature regularization. We also prove the asymptotic stability 
of strictly stable stationary sets. 
\end{abstract}

\maketitle

\tableofcontents
\section{Introduction}

In the last years, the physical
literature has shown a rapidly growing interest toward the study of the morphological instabilities of interfaces between elastic phases generated by the competition between elastic and surface energies, the so called stress driven rearrangement instabilities (SDRI). They occur, for instance, in hetero-epitaxial growth of elastic films when a lattice mismatch between film and substrate is present, or in certain alloys that, under specific
 temperature conditions, undergo a phase separation in many small connected phases
 (that we call particles) within a common elastic body. A third interesting situation is
  represented by the formation and evolution of material voids inside a stressed elastic
   solid. Mathematically, the common thread is that equilibria are identified with local or global minimizers under a volume constraint of a free energy functional, which is given
 by the sum of the stored elastic energy and the surface energy (of isotropic or
 anisotropic perimeter type) accounting for the surface tension along the unknown profile
  of the film or the interface between the phases. The associated variational problems
 can be seen as non-local instances of the {\em isoperimetric principle}, 
 where the non-locality is given by the elastic term. They are very well studied in the
 physical and numerical literature, but the available rigorous mathematical results
 are very few.  We refer to \cite{BGZ15, Bo0,BC, FFLM, FM09, GZ14} for some existence, regularity and stability results related to a  variational model  describing the equilibrium configurations of two-dimensional epitaxially strained elastic films, and to 
\cite{Bo,  CS07} for results in three-dimensions.
 We also mention that a  hierarchy of variational principles to describe the equilibrium shapes in the aforementioned contexts has been introduced in \cite{GuVo}. The simplest prototypical example is perhaps given by the following problem, which can be used to describe the equilibrium shapes of voids in elastically stressed solids (see for instance \cite{siegel-miksis-voorhees04}):
\beq\label{prot1}
\text{minimize }\, J(F):=\frac12\int_{\Om\setminus F}\C E(u_F):E(u_F)\, dz+\int_{\pa F}\vphi(\nu_F)\, d\sigma
\eeq
where minimization takes place among all sets $F\subset\Om$ with  prescribed measure $|F|=m$. Here, the set $F$ represents the shape of the void that has formed within
the  elastic body $\Om$ (an open subset of $\R^2$ or $\R^3$), $u_F$ stands for the equilibrium elastic   displacement in $\Om\setminus F$ subject to a prescribed boundary conditions $u_F=w_0$ on $\pa \Om$ (see \eqref{uF} below), $\C$ is the elasticity tensor 
of the (linearly) elastic material, $E(u_F):=(\nabla u_F+\nabla^T u_F)/2$ denotes the elastic strain of $u_F$, and $\vphi(\nu_F)$ is the anisotropic surface energy density  evaluated at the outer unit normal $\nu_F$ to $F$. The presence of a nontrivial Dirichlet boundary condition $u_F=w_0$ on $\pa \Om$ is what causes the solid $\Om\setminus F$ to be elastically stressed. Indeed, note that when $w_0=0$
the elastic term becomes irrelevant and  \eqref{prot1} reduces to the classical 
Wulff shape problem (with the confinement constraint  $F\subseteq\Om$).  
We refer to \cite{ CJP, FFLMi} for related existence, regularity and stability results in two dimensions.  See also \cite{BCS} for a relaxation result  valid in all dimensions for a variant of \eqref{prot1}.

In this paper we address   the  evolutive counterpart of \eqref{prot1} in two-dimensions, namely  the morphologic evolution  of shapes towards equilbria of the functional $J$, driven by stress and surface diffusion. 
 Assuming that mass transport in the bulk occurs at a much faster time scale,  see \cite{Mu63},  we have, according to the Einstein–Nernst relation, that the evolution is governed by the {\em area preserving} evolution law 
 \beq\label{i1}
V_t=\pa_{\sigma\sigma}\mu_t \qquad\text{on $\pa F(t)$} 
\eeq
where $V_t$ denotes the (outer) normal velocity of the evolving curve $\pa F(t)$ at time $t$ and $\pa_{\sigma\sigma} \mu_t$  stands for the tangential laplacian of the chemical potential $\mu_t$ on $\pa F(t)$. In turn, $\mu_t$ is given by the {\em first variation} of the free-energy functional $J$ at $F(t)$,  and thus (see \eqref{eq:J'} below) \eqref{i1} takes the form 
  \beq\label{i2}
  V_t=\pa_{\sigma\sigma}\bigl(k_{\vphi,t}-Q(E(u_{F(t)}))\bigr)\,, 
  \eeq
where $k_{\vphi,t}$ is the anisotropic curvature of $\pa F(t)$, $u_{F(t)}$ denotes as before the elastic equilibrium in $\Om\setminus F(t)$ subject to
 $u_{F(t)}=w_0$  on $\pa \Om$, and $Q$ is the quadratic form defined as $Q(A):=\frac12\C A: A$ for all $2\times2$-symmetric matrices $A$.
 Note that when $w_0=0$ the elastic term vanishes and thus \eqref{i1} reduces to the {\em surface diffusion  flow } equation
 \beq\label{sdintro}
 V_t=\pa_{\sigma\sigma}k_{\vphi,t} 
 \eeq
  for evolving curves, studied in \cite{EllGarcke} in the isotropic case (see also \cite{EMS} for the $N$-dimensional case).  Thus, we may also regard \eqref{i2} as a sort of  prototypical  nonlocal perturbation   of \eqref{sdintro} by an additive elastic contribution.

 As observed by Cahn and Taylor in the case without elasticity  (see \cite{cahn-taylor94}), the evolution equation \eqref{i2} can be seen as  the gradient flow of the energy functional $J$ with respect to a suitable Riemannian structure of $H^{-1}$ type, see Remark~\ref{rm:accameno1}.

 When the anisotropy $\vphi$ is {\em strongly elliptic}, that is when it satisfies
\beq\label{strongEll}
D^2\vphi(\nu) \,\tau\cdot \tau>0\qquad\text{for all  $\nu\in \mathbb{S}^{1}$ and all $\tau\perp\nu$, $\tau\neq 0$}
\eeq
the evolution \eqref{i2} yields a {\em parabolic } 4-th order (geometric) equation, 
time by time coupled with the {\em elliptic system} describing the elastic equilibrium in $\Om\setminus F(t)$.     

However, we mention here that  for some  physically relevant   anisotropies  the ellipticity condition
\eqref{strongEll}  may fail at some directions $\nu$, see for instance \cite{dicarlo-gurtin-guidugli92, siegel-miksis-voorhees04}. Whenever this happens, \eqref{i2}  becomes {\em backward parabolic} and thus ill-posed.
To overcome this ill-posedness, a canonical approach inspired by  Herring's work   \cite{herring51} consists in considering a {\em regularized curvature-dependent}  surface energy of the form 
$$
\int_{\pa F}\Bigl(\vphi(\nu_F)+\frac\e2 k^2\Bigr)\, d\Ha^1,
$$
where $\e>0$ and $k$ denotes the standard curvature, see \cite{dicarlo-gurtin-guidugli92, GurJab}.
 In this case \eqref{i1} yields the following $6$-th order   area preserving evolution equation
\begin{equation}
V_t=\pa_{\sigma\sigma}\Bigl(k_{\vphi,t}-Q(E(u_{F(t)}))-\e\bigl(\pa_{\sigma\sigma}k+\frac12k^3\bigr)\Bigr)\,.
 \label{sixth order evolution equation}
\end{equation}
This equation was studied numerically in \cite{siegel-miksis-voorhees04} (see also \cite{RRV, BHSV} and references therein)  and analytically in 
\cite{FFLM2}, where local-in-time existence of a solution was established in the context of 
 periodic graphs, modelling the evolution of epitaxially  strained elastic films. We refer also to 
 \cite{FFLM3} for corresponding results in three-dimensions.  We remark that 
 the analysis of \cite{FFLM2} (and of \cite{FFLM3}), which is based on the so-called minimizing movements approach, relies heavily on the presence of the curvature regularization and, in fact, all the estimates provided there are $\e$-dependent and degenerate as $\e\to 0^+$, even when $\vphi$ is strongly elliptic. Thus, the methods developed in \cite{FFLM2, FFLM3} do not apply to the case $\e=0$ in \eqref{sixth order evolution equation}.
 
In this paper we are able to address the case $\e=0$ and in one of the main results (see Theorems~\ref{th:existence}~and~\ref{Cinfty} below) we prove short time existence and uniqueness of a smooth solution of \eqref{i2} starting from sufficiently regular initial sets. To the best of our knowledge this is {\em the first existence result for the surface diffusion flow  with elasticity {and without curvature regularization}}. 
Note that in general one cannot expect global-in-time existence. Indeed, even when no elasticity is present and $\vphi$ is isotropic, singularities such as pinching may develop in finite time, see for instance \cite{GigaIto}.

In the second main result of the paper we establish  global-in-time existence and study the long-time behavior  for a  class of initial data:   we show that {\em strictly stable stationary sets}, that is, sets  $E$ that are stationary  for the energy functional $J$ and  with positive second variation $\pa^2J(E)$ are  {\em exponentially stable}  for the flow \eqref{i2}. More precisely, if the initial set $F_0$ is sufficiently close to the strictly stable set $E$ and has the same area, then the flow \eqref{i2} starting from $F_0$ exists for all times and converges  to $E$ exponentially fast as $t\to+\infty$ (see Theorem~\ref{thmstability} for the precise statement).

A few comments on the strategy of the proofs are in order. The main technical difficulties in proving short-time existence clearly originate  from the presence of the nonlocal elastic term $Q(E(u_{F(t)}))$ in \eqref{i2}. When a curvature regularization as in \eqref{sixth order evolution equation} is present, the elastic term may be regarded and treated as a lower order perturbation and thus is more easily handled. When $\e=0$ this is no longer possible and so one has to find a way to show that the parabolicity of the geometric part of the equation still tames the elastic contribution. Our strategy is based on   
the natural idea of thinking of $Q$ as a {\em forcing term} in order to  set up a fixed point argument.
Roughly speaking, given an initial set $F_0$ and  a forcing term $f$, we let  $t\mapsto F(t)$ be the flow starting from $F_0$ and solving
$$
V_t=\pa_{\sigma\sigma}\bigl(k_{\vphi,t}-f\bigr),
$$
and we consider  the correspondig $t\mapsto Q(E(u_{F(t)}))$, with $u_{F(t)}$ being as usual the elastic equilibrium in $\Om\setminus F(t)$. The existence proof then amounts to finding a fixed point for the map $f\mapsto Q(E(u_{F(\cdot)}))$. 
In order to implement this strategy, the crucial idea is to look at   the squared $L^2$-norm of the tangential gradient of the chemical potential $(k_{\vphi,t}-f)$, that is, to study the behavior of the quantity
\beq\label{monotoneQ}
\int_{\pa F(t)}\bigl(\pa_{\sigma}(k_{\vphi,t}-f)\bigr)^2\, d\H^1
\eeq
with respect to time. More precisely, by computing the time derivative of \eqref{monotoneQ} we  derive suitable energy inequalities involving \eqref{monotoneQ} (see  Lemma~\ref{monotonisuus lemma})  yielding the a priori regularity estimates needed to carry out the  fixed point argument. 
The quantity \eqref{monotoneQ}, with $f$ now given  by the elastic term $Q$,   is also crucial in the aforementioned asymptotic stability analysis. Here, by  adapting to the present situation the methods developed in \cite{AFJM} for the  surface diffusion flow without elasticity, we are able to show that for properly chosen initial sets,  \eqref{monotoneQ} becomes monotone decreasing in time and, in fact, exponentially decays to zero, thus giving the desired exponential convergence result.  
 
This paper is organized as follows.  
 In Section~\ref{sec:preliminaries}  we set up  the problem, introduce the main notations and collect several auxiliary results concerning the energy functional $J$ in \eqref{prot1}. Some of these results, which deal with the properties of strictly stable stationary sets,  are then crucial for the asymptotic stability analysis carried out in Section~\ref{sec:stability}. 
 The short-time existence, uniqueness and regularity of the flow \eqref{i2} for sufficiently regular initial data is addressed in Section~\ref{sec:existence}.
 In Section~\ref{sec:graphs} we briefly illustrate how to apply our main existence and asymptotic stability results in the case of evolving periodic graphs, that is in the geometric setting considered in \cite{FFLM2}. In particular, in Theorem~\ref{th:2dliapunov} we address and analytically characterize the exponential asymptotic stability of  {\em flat configurations}, thus extending to the evolutionary setting the results of \cite{FM09, Bo0}.
 In the final Appendix, for the reader's convenience we provide the proof of an  interpolation result, probably known to the experts,  that is used throughout the paper. 
 
 We conclude this introduction by mentioning that it would be interesting to investigate whether under the assumption \eqref{strongEll} the flows \eqref{sixth order evolution equation} studied in \cite{FFLM2} converge to   \eqref{i2} as $\e\to 0^+$, perhaps using the methods developed in \cite{BMN}. This issue as well as  the extension of the results of this paper to three-dimensions will be addressed in future investigations.

\section{Preliminary results} \label{sec:preliminaries}

\subsection{Geometric preliminaries and notation}

Let $F\subset\R^2$ be a bounded open set of class $C^2$.   We denote the unit outer normal to $F$  by $\nu_F$ and the tangent vector $\tau_F$. Throughout the paper we choose the orientation so that   $\tau_F = \mathcal{R }\nu_F$, where $\mathcal{R}$ is  the counterclockwise rotation by $\pi/2$. 

 The differential of  a vector field $X$ along $\pa F$ 
is denoted by $\pa_\sigma X$. We recall that 
\[
\pa_\sigma \nu_F = k_F \tau_F \qquad \text{and} \qquad \pa_\sigma \tau_F = - k_F \nu_F,
\] 
where $k_F$ is the curvature of $\pa F$.  When  no confusion arises, we will simply write $\nu$, $\tau$,  and $k$ in place of $\nu_F$, $\tau_F$ and $k_F$. The tangential divergence of  $X$  is $\Div_\tau X := \pa_\sigma X  \cdot \tau$. The divergence theorem on $\pa F$ states that  for every vector field $X \in C^1(\pa F; \R^2)$ it holds
\beq\label{divform}
\int_{\pa F}  \Div_\tau X \, d \Ha^1 = \int_{\pa F} k \,  X \cdot \nu \, d \Ha^1.
\eeq

If the boundary of $F$ is of class $C^m$, with $m\geq 2$, then  the signed distance function $d_F$ is of class $C^m$ in a tubular neighborhood of $\pa F$. We may extend $\nu, \tau$ and  $k$ to such a  neighborhood of $\pa F$ by setting  $\nu := \nabla d_F$, $\tau :=\mathcal{R} \nu$ and $k := \Div \nu  = \Delta d_F$.

Throughout the paper, we fix a bounded Lipschitz open set $\Omega \subset \R^2$. Moreover,  $G$ will always denote a smooth reference set, with the property that  $G \subset\subset \Omega$. We will also denote by   $\pi_G$
  the orthogonal projection on $\pa G$ and   by  $\bar\eta$  a positive number such that  
\beq\label{bareta}
\text{$d_G$ and  $\pi_G$ are smooth in  $\mathcal{N}_{\bar\eta}(\pa G)$,}
\eeq
where  $\mathcal{N}_{\bar\eta}(\pa G)$ denotes the $\bar\eta$-tubular neighborhood of $\pa G$. 

 We now introduce a class of sets $F$ sufficiently ``close'' to $G$ so that the boundary can be written as 
\beq\label{bdr}
\partial F = \{ x + h_F(x) \nu_{G}(x) \mid x\in \pa G  \},
\eeq
for a suitable function $h_F$ defined on $\pa G$.
More precisely, for $k\in \N$ and $\alpha\in (0,1)$ we set
\begin{multline}\label{calH}
 \mathfrak{h}^{k,\alpha}_M(G):=\{F\subset\subset\Om:\, \eqref{bdr}\text{ holds for some $h_F\in C^{k,\alpha}(\pa G)$}, \\ \text{with } \|h_F\|_{L^\infty}\leq\bar\eta/2 \text{ and }\|h_F\|_{C^{k,\alpha}} \leq M\}.
\end{multline}
For  such sets $F$ we also  denote by $\pi_{F}^{-1}:\pa G \to \pa F$ the map $\pi_{F}^{-1}(x) = x + h_F(x) \nu_G(x)$
and set
\[
J_F := \sqrt{(1 + h_Fk_G)^2 + (\pa_\sigma h_F)^2},
\]
that is the tangential Jacobian on $\pa G$ of the map $\pi_{F}^{-1}$.
We recall now some useful transformation formulas:
 \beq \label{formula tangent}
\tau_F\circ \pi_{F}^{-1} = \frac{(1+h_F k_G ) \tau_G  + \pa_\sigma h_F   \nu_G }{J_F }
\eeq 
and
\beq \label{formula normal}
\nu_F\circ \pi_{F}^{-1} = \frac{ -\pa_\sigma h_F  \tau_G  + (1+h_F k_G )\nu_G}{J_F}.
\eeq
Similarly, the curvature $k_F$ of $F$  at $y = \pi_{F}^{-1}(x)$ is  given by
\beq \label{curvature formula}
k_F\circ \pi_{F}^{-1}  = \frac{- \pa_{\sigma\sigma}h_F(1+h_F k_G ) + 2 (\pa_\sigma h_F)^2 k_G + (1+h_Fk_G)^2k_G + h_F\pa_\sigma h_F \, \pa_\sigma k_G}{J_F^3} .
\eeq

We now fix some notation, which will be used throughout the paper. If $t\mapsto F_t$ is a (smooth) flow of sets, in order to simplify the notation, we will sometimes write $h_t$, $\nu_t$, $\tau_t$, and $k_t$ instead of $h_{F_t}$,
$\nu_{F_t}$,  $\tau_{F_t}$, and $k_{F_t}$, respectively. Similarly, we will set $k_{\vphi, t}:=g(\nu_t)k_t$.

Moreover, whenever we have a one-parameter family $(g_t)_t$ of functions (or vector fields) we shall denote by $\dot g_t$ the partial derivative with  respect to $t$ of the function 
$(x,t)\mapsto g_t(x)$, and by $\nabla^k g_t$ the $k$-th order differential of the function  $(x,t)\mapsto g_t(x)$ with respect to $x$.

\subsection{The energy functional} 
In this section we introduce the energy functional that underlies the flow. We also introduce the proper notions of stationary points and stability that will be needed in the study of the long-time behavior of the flow, see Section~\ref{sec:stability}.

As explained in the introduction, the free energy functional is the sum of  an anisotropic perimeter and a bulk elastic term.
 
We start by introducing the anisotropic surface energy density, which is given by a positively one-homogeneous function $\vphi \in C^\infty(\R^2\setminus \{0\}; (0,+\infty))$ 

\begin{equation} \label{ellipticity}
D^2 \vphi(\nu) \tau \cdot \tau \geq c_0 >0 
\end{equation}
for every $\nu \in \Sb^1$ and every $\tau \in \Sb^1$ such that $\tau \perp \nu$.  Note that the above condition is equivalent to requiring that the level sets of $\vphi$  have positive   curvature.

Concerning the elastic part, for  $F \subset \! \subset \Omega$ and for the {elastic displacement} $u: \Omega\setminus F\to \R^2$ we denote by $E(u)$ the symmetric part of $\nabla u$, that is, $E(u):= \frac{\nabla u + (\nabla u)^T}{2}$. In what follows,  $\C$ stands for  a fourth order {\em elasticity tensor}   acting on $2\times2$ symmetric matrices $A$, such that $\C A:A>0$ if $A \neq 0$. Finally we shall denote by 
 $Q(A) := \frac{1}{2}\C A : A$ the {\em elastic energy density}. 

We are now ready to write the energy functional. For a fixed {\em boundary displacement}  $w_0\in H^{\frac12}(\pa \Om)$, we set 
\begin{equation} \label{energy}
J(F) :=  \int_{\partial F} \vphi(\nu_F) \, d\Ha^1 + \int_{\Omega \setminus F} Q(E(u_F))\, dx,
\end{equation}
where $u_F$ is the  elastic equilibrium  satisfying the Dirichlet boundary condition $w_0$ on a fixed relatively open subset $\pa_D \Om\subseteq \pa \Om$. More precisely, $u_F$ is the unique solution in  $H^1(\Omega \setminus F; \R^2)$ of the following elliptic system
\beq\label{uF}
\begin{cases}
\Div \C E(u_F)=0 & \text{in }\Om\setminus F,\\
\C E(u_F)[\nu_F]=0 & \text{on }\pa F\cup (\pa \Om\setminus \pa_D\Om),\\
u_F=w_0 &\text {on }\pa_D\Om.
\end{cases}
\eeq

Next, we provide the first and the second variation formulas for \eqref{energy}.
We start by recalling the well-known first variation formula for the anisotropic  perimeter.
To this aim, for any  
 vector field $X \in C_c^1(\R^2; \R^2)$, let  $(\Phi_t)_{t\in (-1,1)}$  be  the associated flow, that is the solution of  
 \beq\label{flussoX}
 \begin{cases}
\displaystyle \frac{\pa \Phi_t}{\pa t}=X(\Phi_t),\\
 \Phi_0=Id.
 \end{cases}
 \eeq
 Then we have
 \[
\frac{d }{d t} \Bigl|_{t=0} \int_{\partial \Phi_t(F)} \vphi(\nu_{\Phi_t(F)}) \, d\Ha^1 = \int_{\partial F} k_\vphi  X\cdot \nu  \, d\Ha^1,
\]
where the {\em anisotropic curvature} $k_\vphi$ of $\pa F$ is given by $k_{\vphi} := \diver_\tau (\nabla \vphi(\nu))$ and can be written also as 
 \[
\begin{split}
k_{\vphi} &= \diver_\tau (\nabla \vphi(\nu)) =  \diver (\nabla \vphi(\nu))=  D^2 \vphi(\nu) : D\nu  \\
&= (D^2 \vphi(\nu) \tau \cdot  \tau)\, k  \\
&=:  g(\nu)\, k,
\end{split}
\]
on $\pa F$. 

Concerning the full functional $J$, we have the following theorem.

\begin{theorem}\label{th:12var}
 Let $F\subset\subset\Om$ be a smooth set, $X \in C_c^1(\Om; \R^2)$ and let $(\Phi_t)_{t\in (-1,1)}$ be the associated flow as in \eqref{flussoX}. Set $\psi:=X\cdot \nu_F$  and $X_\tau:= (X\cdot \tau_F)\tau_F$ on $\pa F$. 
 Then, 
 \beq\label{eq:J'}
\frac{d}{dt}J(\Phi_t(F))_{\bigl|_{t=0}}=\int_{\pa F}(g(\nu_F)k_F-Q(E(u_F))) \psi\, d\Ha^{1}. 
\eeq
If in addition $\Div X=0$ in a neighborhood of $\pa F$ we have
\begin{align}\label{eq:J''}
\frac{d^2}{dt^2}J(\Phi_t(F))_{\bigl|_{t=0}}&=
\int_{\pa F}  \bigl[g(\nu_F) (\pa_{\sigma}\psi)^2-      g(\nu_F) k_F^2  \psi^2\bigr] \, d \Ha^1  - 2 \int_{\Omega \setminus  F} Q(E(u_\psi)) \, dx\nonumber\\ 
&  -  \int_{\pa F}     \pa_{\nu_F} (Q(E(u_F)))  \psi^2 \, d \Ha^1
-\int_{\pa F}(g(\nu_F)k_F-Q(E(u_F)))\Div_\tau(\psi X_\tau)\, d\Ha^1,
\end{align}
where the function $u_\psi$ is the unique solution in $H^1(\Om\setminus F; \R^2)$, with 
$u_\psi=0$ on $\pa_D\Om$, of
\begin{equation}\label{eq u dot2}
\int_{\Om \setminus F} \C E(u_\psi) : E(\vphi) \, dx = -\int_{\pa F}\Div_\tau (\psi E(u_F) ) \cdot \vphi \, d \Ha^1 
\end{equation}
for all $\varphi \in H^1(\Om \setminus F; \R^2)$ such that $\varphi = 0$ on $\pa_D\Omega$. 
\end{theorem}

 Formulas \eqref{eq:J'} and \eqref{eq:J''} have been derived  in \cite{CJP} for the case where $\phi$ is the Euclidean norm, and in a slightly different setting, namely when $F$ is the subgraph of a periodic function,  in  \cite{FM09, Bo}. The very same calculations apply to the more general situation considered here.

Throughout the paper, given a (sufficiently smooth)   set $F \subset \!\subset  \Omega$, we denote by $\Gamma_{F,1}, \dots, \Gamma_{F, m}$ the $m$ connected components of $\pa F$ and by $F_i$ the bounded open set enclosed by $\Gamma_{F,i}$. Note that the $F_i$'s are not in general the connected components of $F$ and it may happen that $F_i\subset F_j$ for some $i \neq j$.  

We are  interested in area preserving variations, in the following sense.
\begin{definition}\label{def:admissibleX}
 Let $F\subset\subset\Om$ be a smooth set.  Given a vector field 
 $X\in C^\infty_c(\Om; \R^2)$, we say that  the associated flow
  $(\Phi_t)_{t\in (-1,1)}$ is {\em admissible for $F$} if there exists $\e_0\in (0,1)$ such that
  $$
  |\Phi_t(F_i)|=|F_i|\quad\text{for $t\in (-\e_0,\e_0)$ and $i=1,\dots, m$.}
  $$
\end{definition}
\begin{remark}\label{rm:mn}
Note that if the flow associated with $X$ is admissible in the sense of the previous definition, then 
for $i=1,\dots, m$ we have 
$$
\int_{\Gamma_{F,i}}X\cdot\nu_F\, d\Ha^1=0.
$$
In view of this remark it is convenient to introduce the space
$\tilde H^1(\pa F)$   consisting of all functions $\psi \in H^1(\pa F)$ with zero average on each component of  $\pa F$, i.e., 
\[
\int_{\Gamma_{F,i}} \psi \, d \Ha^1 = 0 \qquad \text{for every } \, i = 1, \dots, m.
\]
We observe that given $\psi\in \tilde H^1(\pa F)\cap C^{\infty}(\pa F)$ it is possible to construct a  sequence of  vector fields $X_n\in C^\infty_c(\Om; \R^2)$, with $\Div X_n=0$ in a neighborhood of $\pa F$, such that 
$X_n\cdot \nu_F\to \psi$ in $C^1(\pa F)$, see \cite[Proof of Corollary~3.4]{AFM} for the details. Note that in particular   the flows associated with $X_n$ are admissible. 
\end{remark}

\begin{definition}
\label{def stationarity}
Let $F \subset \!\subset  \Omega$  be a set of class $C^2$. We say that $F$ is \emph{stationary} if
$$
\frac{d}{dt}J(\Phi_t(F))_{\bigl|_{t=0}}=0
$$
for all admissible flows in the sense of Definition~\ref{def:admissibleX}.
\end{definition}

\begin{remark}\label{rm:station}
By Remark~\ref{rm:mn} and in view of \eqref{eq:J'} it follows that a set $F \subset \!\subset  \Omega$  of class $C^2$ is stationary if and only if  there exist constants $\lambda_1, \dots, \lambda_m$ such that  
\[
g(\nu_F)k_F - Q(E(u_F)) = \lambda_i \qquad \text{on }\, \Gamma_{F,i}
\]
for every $i = 1,\dots, m$. In turn, note that if $F$ is stationary, then the second variation formula \eqref{eq:J''} reduces to 
\begin{align} \label{sv}
\frac{d^2}{dt^2}J(\Phi_t(F))_{\bigl|_{t=0}}= &
\int_{\pa F}  \bigl[g(\nu_F) (\pa_{\sigma}\psi)^2-      g(\nu_F) k_F^2  \psi^2\bigr] \, d \Ha^1  \nonumber\\
&- 2 \int_{\Omega \setminus  F} Q(E(u_\psi)) \, dx  
  -  \int_{\pa F}     \pa_{\nu_F} (Q(E(u_F)))  \psi^2 \, d \Ha^1,
\end{align}
where we recall that $\psi= X\cdot \nu_F$ and $u_\psi$ is the function satisfying \eqref{eq u dot2}.

Note that if $F$ is a  sufficiently regular (local) minimizer of \eqref{energy} under the constraint 
$|F|=const.$, then  there exists a  constant $\lambda$ such that 
\[
g(\nu_F)k_F - Q(E(u_F)) = \lambda \qquad \text{on }\, \pa F.
\]
Thus, our notion of stationarity differs from the usual notion of criticality just recalled.
\end{remark}
 
 In view of \eqref{sv},  for any set $F\subset\subset\Om$ of class $C^2$ it is convenient to introduce the  quadratic form $\pa^2 J(F)$ defined on $\tilde H^1(\pa F)$ as
\beq \label{eq:pa2J}
\begin{split}
\pa^2J(F)[\psi] :=& \int_{\pa F}  \bigl[g(\nu_F) (\pa_{\sigma}\psi)^2-      g(\nu_F) k_F^2  \psi^2\bigr] \, d \Ha^1  \\
&- 2 \int_{\Omega \setminus  F} Q(E(u_\psi)) \, dx  
  -  \int_{\pa F}     \pa_{\nu_F} (Q(E(u_F)))  \psi^2 \, d \Ha^1,
\end{split}
\eeq  
where $u_\psi$ is the unique solution of  \eqref{eq u dot2} under the Dirichlet condition $u_\psi=0$ on 
$\pa_D\Om$. 
We conclude this section by defining the notion of stability for a stationary point. 
\begin{definition}\label{def:stable}
Let $F\subset\subset\Om$ be a stationary set in the sense of Definition~\ref{def stationarity}. We say that $F$ is \emph{strictly stable} if 
\beq\label{j2>0}
\pa^2J(F)[\psi]>0\qquad\text{for all }\psi\in \tilde H^1(\pa F)\setminus\{0\}.
\eeq
\end{definition}
It is not difficult to see that \eqref{j2>0} is equivalent to the coercivity of $\pa^2J(F)$ on $ \tilde H^1(\pa F)$. More precisely, \eqref{j2>0} holds if and only if there exists $m_0>0$ such that 
\beq\label{emmepiccolo0}
\pa^2J(F)[\psi]\geq m_0\|\psi\|^2_{\tilde H^1(\pa F)}\qquad\text{for all }\psi\in \tilde H^1(\pa F),
\eeq
see \cite{FM09}. In turn the latter coercivity property  is stable  with respect to small $C^{2,\alpha}$-perturbations.  More precisely, we have:
\begin{lemma}\label{lemma:j2>0near}
Assume that the reference set $G\subset\subset\Om$ is  a (smooth) strictly stable stationary set in the sense of Definition~\ref{def:stable} and fix $\alpha\in (0,1)$. Then, there exists  $\sigma_0>0$ such that  for all $F\in \mathfrak{h}^{2,\alpha}_{\sigma_0}(G)$ (see \eqref{calH}) we have
$$
\pa^2J(F)[\psi]\geq \frac{m_0}2\|\psi\|^2_{\tilde H^1(\pa F)} \text{ for all $\psi\in \tilde H^1(\pa  F)$,}
$$
where $m_0$ is the constant in \eqref{emmepiccolo0}.
\end{lemma} 
\begin{proof}The proof of the lemma goes as in \cite[Proof of Lemma~4.12]{FFLM3}, where the case of $F$ being the subgraph of a periodic  function of two variables is considered. Although the geometric framework here is more general, we can follow exactly the same line of argument up to the obvious changes due to the different setting (and some simplifications due the fact that here we work in two-dimensions). We refer the reader to the aforementioned reference for the details. 
\end{proof}
Recall that $G_1$, \dots, $G_m$ are the bounded open sets enclosed by the connected components $\Gamma_{G,1}$, \dots, $\Gamma_{G,m}$ of the boundary $\pa G$ of the reference set and observe
that if $F\in \mathfrak{h}^{2,\alpha}_M(G)$,   then $\pa F$ has the same number $m$ of connected components $\Gamma_{F,1}$, \dots, $\Gamma_{F,m}$, which can be numbered in such a way that 
\beq\label{numbered}
\Gamma_{F,i}=  \{ x + h_{F}(x) \nu_{G}(x) \mid x\in \Gamma_{G,i}  \},
\eeq
 for suitable $h_{F}\in C^{k,\alpha}(\pa G)$.
 
 In the next lemma we show that  pairs of sets which are sufficiently close in a $C^{2,\alpha}$-sense  can always be connected through  area preserving flows in the sense of 
 Definition~\ref{def:admissibleX}. More precisely we have:
 \begin{lemma}\label{lemma:connect}
Let $M>0$ and $\alpha\in (0,1)$. There exists  $C>0$ with the following property: If $F_1$, $F_2\in  \mathfrak{h}^{2,\alpha}_M(G)$ are such that  $|F_{1, i}|=|F_{2,i}|$, $i=1,\dots, m$, then there exists a flow
 $(\Phi_t)_{t\in (-1,1)}$ admissible for $F_1$ in the sense of Definition~\ref{def:admissibleX}, such that $\Phi_0(F_1)=F_1$,  $\Phi_1(F_1)=F_2$,  $|\Phi_t(F_{1,i})|=|F_{1,i}|$ for all $i=1,\dots, m$ and $t\in [0,1]$. Moreover 
 \beq\label{vicinovicino}
 \sup_{t\in [0,1]}\|\Phi_t-Id\|_{C^{2,\alpha}(\mathcal{N}_{{\bar\eta}/2}(\pa G))}\leq C\|h_{F_1}- h_{F_2}\|_{C^{2,\alpha}(\pa G)},
 \eeq
and the velocity field $X$ satisfies $\Div X=0$ in the $\bar \eta/2$-neighborhood $\mathcal{N}_{{\bar\eta}/2}(\pa G)$.
 Here $F_{i, 1}$, \dots, $F_{i,m}$ denote the bounded open sets enclosed by the connected components $\Gamma_{F_i, 1}, \dots, \Gamma_{F_i, m}$ of $\pa F_i$, $i=1,2$, which are supposed to be numbered as in \eqref{numbered}. 
\end{lemma}

\begin{proof}
We adapt the construction of \cite[Proposition~3.4]{Morini}.
We start by constructing a $C^\infty$~vector-field $\tilde X:{\mathcal N}_{\bar \eta}(\pa G)\to\R^2$     satisfying 
\begin{equation}\label{campo1}
\Div \tilde X=0\quad\text{in ${\mathcal N}_{\bar \eta}(\pa G)$},\qquad\quad \tilde X= \nu_G\quad\text{on $\pa G$}.
\end{equation}
To this aim, let $\zeta$ be the solution of 
$$
\begin{cases}
\nabla \zeta\cdot \nabla d_G+\zeta\Delta d_G=0 & \text{in ${\mathcal N}_{\bar \eta}(\pa G)$,}\\
\zeta=1 & \text{on $\pa G$.}
\end{cases}
$$
We may solve the above  PDE by  the method of characteristics, constructing such a $\zeta$ amounts to solving  
for every $x\in\pa G$ the Cauchy problem
$$
\begin{cases}
(f_x)^\prime(t)+f_x(t)\Delta d_G(x+t \nu_G(x))=0 & \text{in $(-\bar\eta,\bar\eta)$,} \\
f_x(0)=1\,,
\end{cases}
$$
and setting 
$$
\zeta(x+t \nu_G(x)):=f_x(t)=\exp\Bigl(-\int_0^t\Delta d_G(x+s \nu_G(x))\, ds\Bigr)\,.
$$
We can  now define   $\tilde X:=\zeta\nabla d_G$ and check that 
$\Div(\zeta\nabla d_G)= \nabla \zeta\cdot \nabla d_G+\zeta\Delta d_G=0$.

 Let now $F_1$ and $F_2$ be as in the statement.
We choose  $X\in C^\infty_c(\Om; \R^2)$  in such a way 
  that  
\beq\label{campo2}
X(x):=\biggl(\int_{h_{F_1}(\pi_G(x))}^{h_{F_2}(\pi_G(x))}\frac{ds}{\zeta(\pi_G(x)+s \nu_G(\pi_G(x)))}\biggr)\,\tilde X(x)\qquad \text{ for every $x\in {\mathcal N}_{\bar \eta/2}(\pa G)$}
\eeq
and we let $\Phi$ be the associated flow.
Notice that the integral appearing in \eqref{campo2} represents the time 
needed to go from $\pi_G(x)+h_{F_1}(\pi_G(x))\nu_G(\pi_G(x))$ to $\pi_G(x)+h_{F_2}(\pi_G(x))\nu_G(\pi_G(x))$ along the trajectory of the vector field $\tilde X$. 
Therefore the above definition ensures that the time needed to go from  $\pi_G(x)+h_{F_1}(\pi_G(x))\nu_G(\pi_G(x))$ to $\pi_G(x)+h_{F_2}(\pi_G(x))\nu_G(\pi_G(x))$ along the modified vector field $X$ is  
one. This is equivalent to saying that for all $x\in \pa G$ we have $\Phi_1(x+h_{F_1}(x)\nu_G(x))= x+h_{F_2}(x)\nu_G(x)$ and, in turn, $\Phi_1(F_1)=F_2$. Moreover, recalling the first equation in \eqref{campo1} and   using the fact that   the function 
$$
x\mapsto\int_{h_{F_1}(\pi_G(x))}^{h_{F_2}(\pi_G(x))}\frac{ds}{\zeta(\pi_G(x)+s \nu_G(\pi_G(x)))}
$$
is constant along the trajectories of $\tilde X$, we deduce from \eqref{campo2} that  the modified field $ X$ is divergence free in  ${\mathcal N}_{\bar \eta/2}(\pa G)$.
  Note that by \eqref{campo2} it also follows 
$$
\|X\|_{C^{2,\alpha}({\mathcal N}_{\bar \eta/2}(\pa G))}\leq C \|h_{F_1}-h_{F_2}\|_{C^{2,\alpha}(\pa G)}
$$
for a constant $C>0$ depending on $G$, and thus \eqref{vicinovicino} easily follows.

Observe now that for $i=1, \dots, m$  and for $\e_0>0$ small enough by \cite[equation (2.30)]{ChSte} we have
$$
\frac{d^2}{dt^2}|\Phi_t(F_{1,i})|=\int_{\Phi_t(\Gamma_{F_i, 1})}(\Div X)(X\cdot \nu_{ \Phi_t(F_{1,i})})=0\,\qquad\text{for all $t\in [-\e_0,1]$,}
$$
where we used the fact that  $ X$ is divergence free in  ${\mathcal N}_{\bar \eta/2}(\pa G)$.
Hence the function $t\mapsto |\Phi_t(F_{1,i})|$ is affine in $[-\e_0,1]$. Since by assumption $|\Phi_0(F_{1,i})|=|F_{1,i}|=|F_{2,i}|=|\Phi_1(F_{1,i})|$, it is in fact  constant. 
This concludes the proof of the lemma.
\end{proof}
We conclude this section by showing that in a sufficiently small $C^{2,\alpha}$-neighborhood of $G$ the stationary sets are isolated, once  we fix the areas enclosed by the connected components of the boundary. 
\begin{proposition}\label{stationary}
Assume that the reference set $G\subset\subset\Om$ is  a (smooth) strictly stable stationary set in the sense of Definition~\ref{def:stable}, fix $\alpha\in (0,1)$, and let $\sigma_0$ be the constant provided by Lemma~\ref{lemma:j2>0near}. There exists $\sigma_1\in (0, \sigma_0)$ with the following property:  Let $F_1$, $F_2\in  \mathfrak{h}^{2,\alpha}_{\sigma_1}(G)$ be stationary sets in the sense of Definition~\ref{def stationarity} and (with same notation of Lemma~\ref{lemma:connect}) assume that $|F_{1, i}|=|F_{2,i}|$ for $i=1,\dots, m$. Then $F_1=F_2$.
\end{proposition}
\begin{proof}
 We start by observing that for any $\eta\in (0,\sigma_0)$ we may choose $\sigma_1>0$ so small that for any pair $F_1$, $F_2\in \mathfrak{h}^{2,\alpha}_{\sigma_1}(G)$ the flow $\Phi_t$ provided by Lemma~\ref{lemma:connect} satisfies
\beq\label{cista}
\Phi_t(F_1)\in \mathfrak{h}^{2,\alpha}_{\eta}(G)\subseteq \mathfrak{h}^{2,\alpha}_{\sigma_0}(G)\qquad\text{for all }t\in [0,1],
\eeq
Notice that this is possible thanks to \eqref{vicinovicino}.  

Recall that by Remark~\ref{rm:station} there exist constants   $\lambda_i$ such that  $g(\nu_G)k_G - Q(E(u_G))=\lambda_i$ on $\Gamma_{G,i}$ for $i=1, \dots, m$. In what follows,  the subscript $t$ stands for the subscript $\Phi_t(F_1)$, where $\Phi_t$ is  the flow of Lemma~\ref{lemma:connect}. Fix $\e>0$ and observe that by taking $\eta$ in \eqref{cista} and, in turn, $\sigma_1$ smaller if needed, we may ensure that 
\beq\label{cista1}
\sup_{t\in [0,1]}\|\nu_t-\nu_G\circ\pi_G\|_{C^1(\Phi_t(\pa F))}\leq \e
\eeq
and 
\beq\label{cista2}
\sup_{i=1,\dots, m}\sup_{t\in [0,1]}\|g(\nu_t) k_t - Q(E(u_t))-\lambda_i\|_{C^0(\Phi_t(\Gamma_{F,i}))}\leq \e.
\eeq
The latter condition can be guaranteed thanks also  to the elliptic estimates proved later in Lemma~\ref{elastiset}. Let $X$ be the velocity field of $\Phi_t$ and recall that by the explicit construction given in the proof of Lemma~\ref{lemma:connect} we have
$X=[X\cdot (\nu_G\circ\pi_G)]\nu_G\circ \pi_G$ in ${\mathcal N}_{\bar \eta/2}(\pa G)$. Thus, writing  
$X=[X\cdot (\nu_G\circ\pi_G-\nu_t)]\nu_G\circ \pi_G+ (X\cdot \nu_t)\nu_G\circ \pi_G$ on 
$\Phi_t(\pa F)$ and using 
\eqref{cista1} with $\e$ (and in turn $\sigma_1$) sufficiently small,  we  find that for all $t\in [0,1]$
\beq\label{cista3}
|X|\leq 2|X\cdot \nu_t|\quad\text{ and }\quad 
|\pa_\sigma X|\leq C\bigl(|X\cdot \nu_t|+|\pa_\sigma(X\cdot \nu_t)|)\bigr)\quad\text{ on }\Phi_t(\pa F),
\eeq
for some constant $C>0$ depending only on $G$.

Let now $F_1$ and $F_2$ be as in the statement of the proposition and $\Phi_t$ as above. Recalling \eqref{eq:J''} and \eqref{eq:pa2J}, for every $s\in[0,1]$
we may write
\begin{align}
\frac{d^2}{dt^2}J(\Phi_t(F))_{\bigl|_{t=s}}&=\pa^2J(\Phi_s(F_1))[X\cdot \nu_s]\nonumber\\
&-\int_{\Phi_s(\pa F_1)}\bigl[g(\nu_s)k_s-Q(E(u_s))\bigr]\Div_\tau\bigl( X_\tau(X\cdot \nu_s)\bigr)\, d\Ha^1\nonumber\\
&=\pa^2J(\Phi_s(F_1))[X\cdot \nu_s]\nonumber\\
&-\sum_{i=1}^m\int_{\Phi_s(\Gamma_{F_1, i})}\bigl[g(\nu_s)k_s-Q(E(u_s))-\lambda_i\bigr]\Div_\tau\bigl( X_\tau(X\cdot \nu_s)\bigr)\, d\Ha^1. \label{catenona}
\end{align}
Recall that $(\Phi_t)$ is an admissible flow and thus  $X\cdot\nu_s\in \tilde H^1(\Phi_s(\pa F_1))$ for every $s\in[0,1]$ due to Remark~\ref{rm:mn}. In turn, by \eqref{cista} and 
Lemma~\ref{lemma:j2>0near} we deduce that 
\[
\pa^2J(\Phi_s(F_1))[X\cdot \nu_s]\geq\frac{m_0}{2}\|X\cdot\nu_s\|^2_{\tilde H^1(\Phi_s(\pa F_1))}.
\] 
Note also that by \eqref{cista3} it is easily checked that
\[
\|\Div_\tau\bigl( X_\tau(X\cdot \nu_s)\bigr)\|_{L^1(\Phi_s(F_1))}\leq C\|X\cdot\nu_s\|^2_{\tilde H^1(\Phi_s(\pa F_1))}.
\]
Thus, collecting all the above observations and recalling also \eqref{cista2}, we deduce from \eqref{catenona} that for every $s\in[0,1]$
\[
\frac{d^2}{dt^2}J(\Phi_t(F))_{\bigl|_{t=s}}\geq \Bigl(\frac{m_0}{2}-Cm\e\Bigr)\|X\cdot\nu_s\|^2_{\tilde H^1(\Phi_s(\pa F_1))}\geq \frac{m_0}{4}\|X\cdot\nu_s\|^2_{\tilde H^1(\Phi_s(\pa F_1))},
\]
where the last inequality holds true provided we choose in \eqref{cista2} a sufficiently small $\e$ (and $\sigma_1$). Since on the other hand by the stationarity of $F_1$ and $F_2$  we have 
\[
\frac{d}{dt}J(\Phi_t(F))_{\bigl|_{t=0}}=\frac{d}{dt}J(\Phi_t(F))_{\bigl|_{t=1}}=0,
\]
 we infer that $\frac{d^2}{dt^2}J(\Phi_t(F))_{\bigl|_{t=s}}=0$ and in turn $X\cdot\nu_s=0$  on $\Phi_s(\pa F_1)$ for all 
$s\in[0,1]$. This means that $s\mapsto \Phi_s(F_1)$ is constant in $[0,1]$ and, in particular, $F_1=F_2$.  
\end{proof}
\section{Short-time existence and regularity}\label{sec:existence}

We are interested in the evolution law 
\begin{equation} \label{flow}
V_t = \pa_{\sigma \sigma}\big(g(\nu_t)k_t - Q(E(u_t))\big) \quad\text{on }\pa F_t,
\end{equation}
where $V_t$ stands for the outer normal velocity of $\pa F_t$, and 
$u_t\in H^1(\Om\setminus F_t; \R^2)$ is the unique solution of  
\beq\label{ut}
\begin{cases}
\Div \C E(u_t)=0 & \text{in }\Om\setminus F_t,\\
\C E(u_t)[\nu_t]=0 & \text{on }\pa F_t\cup (\pa \Om\setminus \pa_D\Om),\\
u_t=w_0 &\text {on }\pa_D\Om.
\end{cases}
\eeq

\begin{remark}\label{rm:accameno1}
We remark that \eqref{flow} can be regarded as the gradient flow of \eqref{energy} with respect to a suitable Riemannian structure of $H^{-1}$-type.
To illustrate this fact, consider the dual $H^{-1}_t$  of $H^1_t:=H^1(\pa F(t)$) endowed with the scalar product  
\begin{multline}\label{hmeno1scalar}
\langle \psi_1, \psi_2\rangle_{H^{-1}_t}:=\int_{\pa F(t)}\pa_\sigma v_{\psi_1}\pa_\sigma v_{\psi_2}\,d\H^1=-\langle \pa_{\sigma\sigma}v_{\psi_2}, 
v_{\psi_1}\rangle_{H^{-1}_t\times H^{1}_t}\\ =
\langle \psi_2, 
v_{\psi_1}\rangle_{H^{-1}_t\times H^{1}_t}=\langle \psi_1, 
v_{\psi_2}\rangle_{H^{-1}_t\times H^{1}_t},
\end{multline}
where  $\pa_\sigma$ denotes the tangential derivative on $\pa F(t)$ and for any $\psi\in H^{-1}_t$ the function $v_\psi$ is the unique function  in  $H^1_t$ satisfying  
\beq\label{vpsiintro}
\begin{cases}
-\pa_{\sigma\sigma}v_\psi= \psi & \text{on $\pa F(t)$}\,, \vspace{5pt}\\
\displaystyle \int_{\pa F(t)} v_\psi\, d\mathcal{H}^1=0\,.
\end{cases}
\eeq
As already recalled, the first variation $\pa J(F(t))$ satisfies
$$
\pa J(F(t))[\psi]=\int_{\pa F(t)} \bigl(k_{\vphi, t}-Q(E(u_{F(t)}))\bigr) \psi\, d\H^1
$$
for all $\psi\in C^\infty(\pa F(t))$.
Thus, recalling also \eqref{flow},  \eqref{hmeno1scalar} and \eqref{vpsiintro}, we  have
\begin{align*}
\langle V_t, \psi \rangle_{H^{-1}(\pa F(t))}&=\int_{\pa F(t)}V_t\,  v_{\psi}\, d\H^1
=\int_{\pa F(t)} \pa _{\sigma\sigma}\bigl(k_{\vphi, t}-Q(E(u_{F(t)}))\bigr) v_{\psi}\, d\H^1\\
&= \int_{\pa F(t)} \bigl(k_{\vphi, t}-Q(E(u_{F(t)}))\bigr) \pa_{\sigma\sigma}v_{\psi}\, d\H^1=-\pa J(F(t))[\psi].
\end{align*}
Hence, time by time the normal velocity $V_t$  is the element of $H^{-1}_t$ that represents the action of  $-\pa J(F(t))$ with respect to the scalar product defined in \eqref{hmeno1scalar}. This formally establishes the $H^{-1}$-gradient flow structure of \eqref{flow}.
\end{remark}

The following theorem establishes the  short-time existence of a unique weak solution of  \eqref{flow}. 
In Theorem~\ref{Cinfty} below we will show that  the weak solution is in fact smooth and therefore classical.

\begin{theorem}\label{th:existence}
Let $F_0\subset\subset\Om$ be such that  
\beq\label{f0}
\partial F_0 = \{ x + h_0(x) \nu_{G}(x) \mid h_0 \in  H^3(\partial G)  \}.
\eeq
There exist $\delta$ and $T>0$, which depend on  the $H^3$-norm of $h_0$, such that if $\|h_0\|_{L^2(\pa G)} \leq \delta$ then  the flow \eqref{flow} admits a unique local-in-time weak solution 
$(F_t)_{t \in (0,T)}$ with an initial set $F_0$. 
More precisely, we have $\partial F_t = \{ x + h_t(x) \nu_{G}(x) \}$, where   
$(h_t)_t\in  H^1(0,T; H^1(\pa G))\cap L^2(0, T; H^3(\pa G))$.   
Moreover $(h_t)_t \in C^0([0, T); C^{2,\alpha}(\pa G))$  for all $\alpha\in (0, \frac12)$  and 
$\bigl([g(\nu_t)k_t-Q(E(u_t))]\circ\pi_{F_t}^{-1}\bigr)_t \in L^2(0,T; H^3(\pa G))$.
\end{theorem}
Note that when the initial set $F_0$ is smooth we may take $G=F_0$. 
We give the proof of Theorem \ref{th:existence} at the end of the section. We will  first prove a sequence of lemmas needed for the proof of the short-time existence result.

We will need some preliminary results.  Our proof of Theorem \ref{th:existence} is based on a fixed point argument. To this aim, for a given smooth function $f : \pa G \times (0,T) \to \R$,  we consider the forced surface diffusion flow given by
\beq\label{eqf}
V_t = \pa_{\sigma \sigma}\big(g(\nu_t)k_t + f_t \circ \pi_G\big)\quad\text{on }\pa F_t
\eeq
with initial datum $F_0$  of class $H^3$, where we denoted $f_t:=f(\cdot, t)$. To simplify the notation we will denote 
\beq\label{Rt}
R_t:=g(\nu_t)k_t + f_t \circ \pi_G \quad\text{on }\pa F_t.
\eeq

The following  monotone quantities are the starting point of our analysis. 
\begin{lemma} \label{monotonisuus lemma}
Let $F_0$ be a set with smooth boundary, $f  \in C^{\infty}(\pa G\times [0,\infty))$, and let  $(F_t)_{t \in (0,T)}$ be a smooth flow  satisfying \eqref{eqf}, with initial datum $F_0$. Then we have
\beq\label{stimaf3}
\frac{d}{dt}\int_{F_t\Delta G} \text{dist}(x, \pa G)\, dx = \int_{\pa F_t} d_{G}\, \pa_{\sigma\sigma}R_{t}\, d\Ha^1 \leq  P(F_t)^{\frac{1}{2}}  \left( \int_{\pa F_t}(\pa_\sigma R_t)^2 \, d\Ha^1\right)^{\frac{1}{2}},
\eeq
whenever the flow \eqref{eqf} exists.  Moreover there exists $C_1$, which depends on $\sup_{(0,T)} ||h_t||_{C^{2,\alpha}}$ and $\sup_{(0,T)} ||f_t||_{C^{1,\alpha}}$,  such that 
\begin{multline}\label{stimaf2}
\frac{d}{dt}\int_{\pa F_t}(\pa_\sigma R_t)^2\, d\Ha^1+
c_0\int_{\pa F_t} (\pa_{\sigma\sigma\sigma}R_t)^2\, d\Ha^1 \\
\leq  C_1\| \dot f_t  \|_{H^{-\frac12}(\pa G)}^2 + C_1\left( 1+ \int_{\pa F_t}(\pa_\sigma R_t)^2\, d\Ha^1\right)^q\,,
\end{multline}
for some $q>1$. 
\end{lemma}

\begin{proof}
 Let   $X_t$ be the  velocity field associated with the flow. In particular we have that 
\[
X_t \cdot \nu_{t} = \pa_{\sigma \sigma}R_t. 
\]
For  $t \in (0,T)$  and $s>0$  $\Phi_s: \pa F_t  \to \pa F_{t+s}, \Phi_s = \pi_{F_{t+s}}^{-1} \circ \pi_{F_t}$ are admissible  diffeomorphisms and by the above equality it holds $\frac{\pa}{\pa s} \Phi_s \big|_{s=0} = X_t$.

 As mentioned in the previous section we can extend  $\nu_t, \tau_t$ and $k_t$ by means of the signed distance function $d_{F_t}$ in a tubular neighborhood of $\pa F_t$. This extension yields the following identities (see for instance \cite[Lemma 4.2]{Bo}):
\beq
\label{app1}
\pa_{\nu_t} k_{\varphi, t} = - k_t^2 g(\nu_{t}),
\eeq
\beq
\label{app2}
\dot{\nu}_t =   - \pa_\sigma (X_t \cdot \nu_t) \, \tau_t =  - \pa_{\sigma \sigma \sigma} R_t\, \tau_t
\eeq
and
\beq
\label{app3}
\dot{k}_{\varphi, t}  = \Div (D^2 \varphi(\nu_t)  \,\dot{\nu}_t ) =  - \pa_\sigma (g(\nu_{t})  \pa_{\sigma \sigma \sigma}R_t). 
\eeq

Note that 
\[
\int_{F_t \Delta G} \text{dist}(x, \pa G)\, dx = \int_{F_t} d_G\, dx  - \int_{G}d_G\, dx .
\]
Thus, 
\[
\begin{split}
\frac{d}{dt} \int_{F_t \Delta G} \text{dist}(x, \pa G)\, dx  &= \frac{d}{dt}\int_{F_t} d_G\, dx = \int_{F_t} \Div( d_G X_t) \, dx \\
&=   \int_{\pa F_t}  d_G (X_t\cdot \nu_t ) \, d \Ha^1  =  \int_{\pa F_t}    d_G \, \pa_{\sigma \sigma}R_t\, d \Ha^1\\
&= - \int_{\pa F_t}    \pa_{\sigma}d_G   \, \pa_{\sigma}R_t\, d \Ha^1 \leq P(F_t)^{\frac12} \left(   \int_{\pa F_t}  \, (\pa_{\sigma}R_{t})^2 \, d \Ha^1 \right)^{\frac12}.
\end{split}
\]
This proves \eqref{stimaf3}.
Concerning \eqref{stimaf2} we begin by  calculating  
\beq
\label{app4}
\begin{split}
\frac{d}{dt} &\left(\frac{1}{2} \int_{\pa F_{t}} (\pa_{\sigma}R_{t})^2  \, d \Ha^1\right) =  \frac{\pa}{\pa s} \left(\frac{1}{2} \int_{\pa F_{t}} \big( (\nabla R_{t+s})(\Phi_s(x)) \cdot \tau_{t+s}( \Phi_s(x)) \big)^2  \, J_{\tau} \Phi_s \, d \Ha^1\right) \big|_{s=0}\\
&= \frac{1}{2} \int_{\pa F_{t}} (\pa_{\sigma}R_{t})^2  \Div_\tau X_t \, d \Ha^1 + \int_{\pa F_{t}} \pa_{\sigma}R_{t} \, \frac{\pa}{\pa s} \Big(\nabla R_{t+s}(\Phi_s(x)) \cdot  \tau_{t+s}( \Phi_s(x))   \Big) \big|_{s=0} \, d \Ha^1.
\end{split}
\eeq
Using our notation we write the last term as 
\[
\begin{split}
\frac{\pa}{\pa s} & \Big(\nabla R_{t+s}(\Phi_s(x)) \cdot  \tau_{t+s}( \Phi_s(x))   \Big) \big|_{s=0}\\
&= \pa_{\sigma} \dot{R_t} +  (\nabla^2R_t \, X_t)\cdot \tau_t    + \nabla R_t \cdot  \dot{\tau_t} + \nabla R_t\cdot( \nabla \tau_t X_t).
\end{split}
\]
We write $X_{t, \tau} := X_t \cdot \tau_t  $.
 Note that by \eqref{app2} we have that $\dot{\tau}_t = \mathcal{R}\dot{\nu}_t =  \pa_{\sigma \sigma \sigma}R_t \, \nu_t$. Moreover it holds 
$D\tau_t \, \nu_t = 0.$ Therefore we get
\[
\begin{split}
\frac{\pa}{\pa s} & \Big(\nabla R_{t+s}(\Phi_s(x)) \cdot  \tau_{t+s}( \Phi_s(x))   \Big) \big|_{s=0} \\
&= \pa_{\sigma} \dot{R_t} +  \pa_{\sigma \sigma \sigma} R_t\,  \pa_{\nu_t} R_t + \pa_{\sigma \sigma}R_t   (\nabla^2R_t \, \nu_t)\cdot \tau_t  +    (\nabla^2R_t \, \tau_t)\cdot \tau_t  \, X _{t,\tau}    + \nabla R_t \cdot(\nabla  \tau_t \, \tau_t)\, X _{t,\tau}.
\end{split}
\]
Therefore,  using the fact that $\pa_\sigma (\pa_{\nu_t} R_t)  = k_t \, \pa_{\sigma}R_t +  
(\nabla^2R_t \, \nu_t)\cdot \tau_t $ and integrating by parts, \eqref{app4} can be written as
\[
\begin{split}
\frac{d}{dt} &\left(\frac{1}{2} \int_{\pa F_{t}} (\pa_{\sigma}R_{t})^2  \, d \Ha^1\right) =   \int_{\pa F_{t}}  \frac{1}{2}(\pa_{\sigma}R_{t})^2  \Div_\tau X_t + \pa_{\sigma}R_{t}  \, \pa_\sigma\dot{R_t} +  \pa_{\sigma}R_{t}  \, \pa_{\sigma \sigma \sigma}R_t\,  \pa_{\nu_t} R_t  \, d \Ha^1\\
&+\int_{\pa F_{t}}\pa_{\sigma}R_{t}  \,  \pa_{\sigma \sigma}R_t  (\nabla^2R_t \, \nu_t)\cdot \tau_t   +\pa_{\sigma}R_{t} \, X _{t,\tau} \,   (\nabla^2R_t \, \tau_t)\cdot \tau_t   + \pa_{\sigma}R_{t} \, \nabla R_t\cdot (\nabla \tau_t\, \tau_t) \, X _{t,\tau}  \, d \Ha^1\\
&= \int_{\pa F_{t}} \frac{1}{2}(\pa_{\sigma}R_{t})^2  \Div_\tau X_t   -\pa_{\sigma \sigma}R_t \dot{R_t}  -  (\pa_{\sigma \sigma}R_t)^2  \pa_{\nu_t} R_t  - k_t \, (\pa_{\sigma}R_{t})^2  \pa_{\sigma \sigma}R_t \, d \Ha^1\\
&+ \int_{\pa F_{t}}   \pa_{\sigma}R_{t} \,  (\nabla^2R_t \, \tau_t)\cdot \tau_t \, X _{t,\tau}     +\pa_{\sigma}R_{t}  \nabla R_t\cdot (\nabla \tau_t\, \tau_t) \, X _{t,\tau}     \, d \Ha^1.
\end{split}
\]

Note that 
\[
\frac{1}{2}\Div_\tau\big((\pa_\sigma R_t)^2 X_t\big) =\frac{1}{2}(\pa_{\sigma}R_{t})^2  \Div_\tau X_t +\pa_{\sigma}R_{t} \,  (\nabla^2R_t \, \tau_t)\cdot \tau_t \, X _{t,\tau}     +\pa_{\sigma}R_{t}  \nabla R_t\cdot (\nabla \tau_t\, \tau_t) \, X _{t,\tau}
 \]
Hence, using also \eqref{divform}, we get
\beq \label{app5}
\begin{split}
\frac{d}{dt} &\left(\frac{1}{2} \int_{\pa F_{t}} (\pa_{\sigma}R_{t})^2  \, d \Ha^1\right)\\
&=  \int_{\pa F_{t}} \frac{1}{2}\Div_\tau\big((\pa_\sigma R_t)^2 X_t\big)  \, d \Ha^1 
 -\pa_{\sigma \sigma}R_t \dot{R_t}  -  (\pa_{\sigma \sigma}R_t)^2  \pa_{\nu_t} R_t  - k_t \, (\pa_{\sigma}R_{t})^2  \pa_{\sigma \sigma}R_t \, d \Ha^1\\
&=  - \int_{\pa F_{t}} \pa_{\sigma \sigma}R_t \dot{R_t}  + (\pa_{\sigma \sigma}R_t)^2  \pa_{\nu_t} R_t  +\frac12 k_t \, (\pa_{\sigma}R_{t})^2  \pa_{\sigma \sigma}R_t  \, d \Ha^1.
\end{split}
\eeq

Therefore, recalling \eqref{Rt}, by  \eqref{app1} and  \eqref{app3} we get from \eqref{app5} that

\begin{multline} \label{app6}
\frac{d}{dt} \left(\frac{1}{2} \int_{\pa F_{t}} (\pa_{\sigma}R_{t})^2  \, d \Ha^1\right)= -  \int_{\pa F_{t}}  g(\nu_t) (\pa_{\sigma \sigma \sigma}R_t)^2\, d \Ha^1  - \int_{\pa F_{t}}   \pa_{\sigma \sigma }R_t ( \dot f_t \circ \pi_G )\, d \Ha^1 \\ 
-  \int_{\pa F_{t}} \Big(  \pa_{\nu_t} (f_t \circ \pi_G)  (\pa_{\sigma \sigma} R_t)^2  - g(\nu_t) k_t^2 (\pa_{\sigma \sigma} R_t)^2  +  \frac{1}{2}   k_t \, (\pa_{ \sigma} R_t)^2 \pa_{\sigma \sigma} R_t \Big)  \, d \Ha^1.
\end{multline}
By the ellipticity assumptions \eqref{ellipticity} we have that $c_0 \leq g(\nu_t) \leq C_0$ and $|k_t| \leq \frac{1}{c_0} |k_{\varphi,t}| \leq C$, where $C$ depends also on the $C^{2,\alpha}$-norm of $h_t$. For $\e>0$ to be chosen, using also Young's inequality, we may  estimate \eqref{app6} as
\[
\begin{split}
\frac{d}{dt} &\left(\frac{1}{2} \int_{\pa F_{t}} (\pa_{\sigma}R_{t})^2  \, d \Ha^1\right)  + c_0\int_{\pa F_{t}}   (\pa_{\sigma \sigma \sigma}R_t)^2\, d \Ha^1  \\
 &\leq   C_\eps\| \dot f_t\|_{H^{-\frac12}(\pa G)}^2  +  \eps ||\pa_{\sigma \sigma}R_t||_{H^{\frac12}(\pa F_t)}^2 + C \int_{\pa F_{t}} \left( 1+ (\pa_{\sigma \sigma}R_t)^2  + (\pa_\sigma R_t)^2 |\pa_{\sigma \sigma}R_t|  \right)  \, d \Ha^1,
\end{split}
\]
where the constant $C$ depends on the $C^{2,\alpha}$-norm of $h_t$ and the $C^{1,\alpha}$-norm of $f_t$.  Since
$\|\pa_{\sigma \sigma}R_t\|_{H^{\frac12}(\pa F_t)}\leq C \|\pa_{\sigma \sigma \sigma}R_t\|_{L^{2}(\pa F_t)}$,   by choosing $\eps$ small enough we get
\[
\begin{split}
\frac{d}{dt} &\left(\frac{1}{2} \int_{\pa F_{t}} (\pa_{\sigma}R_{t})^2  \, d \Ha^1\right)  + \frac23c_0\int_{\pa F_{t}}   (\pa_{\sigma \sigma \sigma}R_t)^2\, d \Ha^1   \\
 &\leq  C\| \dot f_t\|_{H^{-\frac12}(\pa G)}^2   + C \int_{\pa F_{t}} \left( 1+ (\pa_{\sigma \sigma}R_t)^2  + (\pa_\sigma R_t)^4 \right)  \, d \Ha^1.
\end{split}
\]
 Note now that by  Theorem~\ref{interpolation}, 
\[
\|\pa_{\sigma \sigma} R_t\|_{L^{2}}^2 \leq C \|\pa_{\sigma \sigma \sigma} R_t\|_{L^{2}}\,\|\pa_\sigma R_t\|_{L^{2}} \leq \eps  \|\pa_{\sigma \sigma \sigma} R_t\|_{L^{2}}^2 + \frac{C}{\eps} \|\pa_{ \sigma} R_t\|_{L^{2}}^2
\]
and
\[
\|\pa_{ \sigma} R_t\|_{L^{4}}^4 \leq C \|\pa_{ \sigma\sigma\sigma } R_t\|_{L^{2}}^{\frac12}\|\pa_{ \sigma} R_t\|_{L^{2}}^{\frac72} \leq \eps  \|\pa_{ \sigma\sigma\sigma} R_t\|_{L^{2}}^2 + \frac{C}{\eps} \|\pa_{ \sigma} R_t\|_{L^{2}}^{\frac{14}{3}}.
\]
Hence the estimate \eqref{stimaf2} follows for $q = \frac{7}{3}$ by choosing $\eps$ small enough. 
\end{proof}
Next theorem establishes the local in time existence for \eqref{eqf}. The same result for $f=0$ is proved in \cite{chinese}. The case considered here follows essentially from the same argument.
 \begin{theorem}\label{th:probabile}
Let $h_0\in H^4(\pa G)$ and let $f\in C^\infty(\pa G\times [0,+\infty))$.  Then, there exist $\de>0$ and  $T>0$ such that if $\|h_0\|_{C^1(\pa G))}\leq\de$, then \eqref{eqf} admits a smooth solution $(F_t)_t$ defined for all $t\in (0,T)$. Moreover, setting  
$\partial F_t = \{ x + h_t(x) \nu_{G}(x) \}$, we have that $(h_t)_t\in H^1(0,T; L^2(\pa G))\cap L^2(0,T; H^4(\pa G))$. Finally, there exists $\bar \de\in (0, \bar \eta)$, where $\bar \eta$ is as in \eqref{bareta}, depending only on $G$ and $\Om$, such that  if $\sup_{(0,T)}\|h_t-h_0\|_{C^1(\pa G)}< \bar \de$, then the solution can be extended beyond the time $T$.  
\end{theorem}
\begin{proof}
The proof goes exactly   as the one   of Theorem~2.5 of \cite{chinese}, taking into account the presence of the forcing term $f$.   Note that as in \cite{chinese} in the first part of the proof we can only conclude that the time $T$ depends on $\|h_0\|_{H^4}$ and on $\|f\|_{L^2(0,T; H^2)}$. However, one can then argue as in the second part of  proof of Theorem~2.5 of \cite{chinese} to conclude that the $\bar\de$ for which the extension property holds is independent of  $\|h_0\|_{H^4}$ and  $\|f\|_{L^2(0,T; H^2)}$, as long as 
$f\in L^2(0,T; H^2)$ (a property which is implied by our assumption on $f$). Finally, the $C^\infty$-regularity of the solution for $t>0$ follows by standard arguments (or arguing as in the proof of Theorem~\ref{Cinfty} below where in fact the more complicated equation \eqref{flow} is dealt with). 
\end{proof}
In the next lemma we use the monotone quantity \eqref{stimaf2} to deduce regularity estimates for the flow \eqref{eqf}.  
\begin{lemma} \label{flow exists 1}
Let $F_0$ be a smooth initial set  and $h_0 \in C^\infty(\pa G)$ the function representing $\pa F_0$ as in \eqref{f0}.  Fix   $M_0>0$, $\alpha \in (0,\tfrac12)$, $\de_1\in (0,\bar \de)$, with $\bar\de$ as in 
Theorem~\ref{th:probabile}. There exist $\delta_0>0$     and $T_0$, depending on $M_0$, $\alpha$,  and $\de_1$, such that if  $f \in C^\infty(\pa G\times [0,+\infty))$  satisfies 
\beq\label{flex0}
\sup_{(0,T_0)} \|f_t\|_{C^{1, \alpha}(\pa G)} \leq M_0\qquad \text{and} \qquad \int_0^{T_0} \|\dot f_t\|_{H^{-\frac12}(\pa G)}^2 \leq M_0
\eeq
 and if   $\|h_0\|_{H^3(\pa G)} \leq M_0$, $\|h_0\|_{L^{2}(\pa G)} < \delta_0$, then the flow \eqref{eqf} exists on $(0,T_0)$ and
\beq \label{flex1}
\sup_{(0,T_0)}\|h_t\|_{C^{2, \alpha}(\pa G)} \leq \delta_1 \qquad \text{and} \qquad \sup_{(0,T_0)} \|\pa_\sigma R_t\|_{L^{2}(\pa F_t)}^2 \leq 2C_1 M_0 + \|\pa_\sigma R_0\|_{L^{2}(\pa F_0)}^2\,,
\eeq
where $C_1$ is the constant appearing in Lemma~\ref{monotonisuus lemma}.
\end{lemma}

\begin{proof}
 We fix $\delta_1 < \bar \de$ and observe if $\de_0>0$ is sufficiently small,   $\|h_0\|_{H^3(\pa G)} \leq M_0$  and  $\|h_0\|_{L^{2}(\pa G)} < \delta_0$, then from \eqref{inter3} we get $\|h_0\|_{C^{2, \alpha}(\pa G)}  < \bar\de-\delta_1$. In particular, by Theorem~\ref{th:probabile} the flow exists for a short time and as long as $\|h_t\|_{C^{2, \alpha}(\pa G)} < \delta_1$, since this implies that $\|h_t-h_0\|_{C^1(\pa G)}<\bar\de$. 
 
 Let us denote by $T_0$ the maximal time  such that 
 \beq\label{flex2}
 \|h_t\|_{C^{2, \alpha}(\pa G)} < \delta_1 \text{ and }  \|\pa_\sigma R_t\|_{L^{2}(\pa F_t)}^2 < 2C_1 M_0 + \|\pa_\sigma R_0\|_{L^{2}(\pa F_0)}^2 \quad \text{for all }t\in (0, T_0)\,.
 \eeq
We want to show that if \eqref{flex0} is satisfied, then $T_0$ is bounded away from 0 by a constant depending only $M_0$, $\alpha$, and $\de_1$. Thus, without loss of generality we may assume that 
$T_0\leq 1$, otherwise there is nothing to prove. 

Assume first that $\|h_{T_0}\|_{C^{2, \alpha}(\pa G)} =\delta_1$.  From \eqref{flex2},   from the first inequality in \eqref{flex0} and from the assumption  $\|h_0\|_{H^3(\pa G)} \leq M_0$    we conclude that  
\begin{equation}\label{stop est 00}
\|\pa_\sigma R_t\|_{L^2(\pa F_t)} \leq C(M_0) \qquad \text{for all } \, t \in (0, T_0).
\end{equation}
In turn, using again the first inequality in \eqref{flex0}   we get 
\[
  \|\pa_\sigma k_{\vphi,t}\|_{L^2(\pa F_t)}^2 \leq  C(M_0) \qquad \text{for all } \,  t \in (0, T_0).
\]
Now, by the first inequality in \eqref{flex2}, recalling also formula \eqref{curvature formula}, we deduce that 
\begin{equation}\label{stop est 0}
  \|h_t\|_{H^3(\pa G)} \leq C(M_0) \qquad \text{for all } \, t \in (0, T_0).
\end{equation}
Moreover, we conclude  by integrating  \eqref{stimaf3} over $(0,t)$ and  by \eqref{stop est 00}  that 
$$
  \|h_t\|_{L^2(\pa G)}^2 \leq C(M_0)T_0 + C\|h_0\|_{L^2(\pa G)}^2\leq C(M_0)T_0 + C \delta_0^2\qquad \text{for all } \, t \in (0, T_0)\,. 
$$
In turn, by \eqref{inter3} and by \eqref{stop est 0} we get 
\begin{align*}
\delta_1  =  \|h_{T_0}\|_{C^{2, \alpha}(\pa G)} \leq C  \, \bigl(\|h_{T_0}\|_{H^{3}(\pa G)}^{\theta} \, \|h_{T_0}\|_{L^{2}(\pa G)}^{1-\theta}+ \|h_{T_0}\|_{L^{2}(\pa G)}\bigr)\leq C(M_0)\left(\sqrt{T_0}+\delta_0\right)^{1-\theta}\,,    
\end{align*}
where $\theta$ depends only on $\alpha$. It is clear from the above inequality that if $\de_0$ is sufficiently small  we get that $T_0$ is bounded away from $0$.

Assume now that $y(T_0)=2C_1M_0+ y(0)$, where we have set $y(t):=||R_\sigma||_{L^{2}(\pa F_t)}^2$.
 By integrating \eqref{stimaf2}  over the time interval $(0,T_0)$ and using the second inequality in \eqref{flex0} we get
\[
y(T_0) \leq y(0) + C_1M_0 + C_1 \int_0^{T_0}(1 + y)^q \, dt. 
\] 
Now, using the second inequality in \eqref{flex2} we conclude that 
$$
2C_1M_0+ y(0)=y(T_0)\leq y(0) + C_1M_0 + C_1T_0 (1+  2C_1 M_0 + y(0))^q.
$$
From this estimate we get again that $T_0$ is bounded away from $0$. This concludes the proof of the lemma.
\end{proof}

We will need the following result for the elastic equilibrium, which states that  if $F$, $\widetilde F \in \mathfrak{h}^{2,\alpha}_M(G)$ are $C^{2, \alpha}$-close, then  the corresponding elastic equilibria  are $C^{2,\alpha}$-close to each other. More precisely, we have the following lemma.

\begin{lemma} \label{elastiset}
Let $0<\alpha<\beta\leq1$, $M>0$ and $k\in \N$.  Then there exists  $C>0$  such that if  $F, \widetilde F  \in \mathfrak{h}^{k,\beta}_M(G)$ we have
\beq\label{linear}
\| u_{F} \circ \pi_F^{-1} - u_{\widetilde F} \circ \pi_{\widetilde F}^{-1}\|_{C^{k,\alpha}(\pa G)} \leq C \| h_{F}-h_{\widetilde F}\|_{C^{k,\alpha}(\pa G)}.
\eeq 
Here, we recall that   $u_{F}$ and $u_{\widetilde F}$ denote the elastic equilibria 
corresponding to $F$ and $\tilde F$, respectively, as defined in  \eqref{uF}.
\end{lemma}
\begin{proof} The case $k=1$ can be proved as in  \cite[Lemma~6.10]{FFLM2}. We now assume $k\geq 2$. 

Denote by $\mathcal U$ the open set in $C^{k,\alpha}(\pa G)$ defined as
$$
\mathcal U:=\big\{h\in C^{k,\alpha}(\pa G):\,\,\|h\|_{L^\infty(\pa G)}<2\bar\eta/3,\,\,\|h\|_{C^{k,\alpha}(\pa G)}<M'\big.\},
$$
where $M'>0$ is chosen so large that $\mathfrak{h}^{k,\beta}_M(G)\subset\mathcal U$.

Given $h\in \mathcal U$, we denote by $u_h$ the solution to \eqref{uF}, with 
$F$ replaced by the bounded  set $F_h$ whose boundary is given by the (normal) graph of $h$ over $\pa G$.

Fix  $\psi\in C^\infty_c(\Om)$, $0\leq\psi\leq 1$, $\psi\equiv 1$ in $\{d_G\leq 2\bar\eta/3\}$, and $\mathrm{supp\,}\psi\subset\subset \{d_G< \bar\eta\}$ and notice that if $\| h\|_{C^{k,\alpha}(\pa G)}\leq \de'$, for $\de'$ sufficiently small, then the map $\Phi_h:\Om\setminus F_h\to \Om\setminus G$ of the form
$$
\Phi_h(x)=x-h(\pi_G(x))\psi(x)\nu_G(\pi_G(x))
$$
is a $C^{k,\alpha}$-diffeomorphism.

 Then setting $v_h:=u_h\circ(\Phi_h)^{-1}$ one can see that  $v_h$ is the solution to 
$$
\begin{cases}
\Div(\mathbb{A}(y, h(\pi_G(y)), \pa_\sigma h(\pi_G(y)))\nabla v)=0 & \text{in $\Omega\setminus G$}\\
\mathbb{A}(y, h(y), \pa_\sigma h(y))\nabla v[\nu_G]=0 & \text{on $\pa G$,}\\
v=w_0 & \text{on $\pa_D\Om$,}
\end{cases}
$$
where the entries of the tensor valued function $\mathbb{A}$ are 4-th order polynomials in $h\circ\pi$ and $\pa_\sigma h\circ\pi$ with $C^{k-1,\alpha}$-coefficients. It is easily checked that the map $\mathcal F:\mathcal U\times C^{k,\alpha}(\Om\setminus G)\to 
C^{k-2,\alpha}(\Om\setminus G)\times C^{k-1,\alpha}(\pa G)$ given by
$$
\mathcal F(h, v):=(\Div(\mathbb{A}(y, h(\pi_G(y)), \pa_\sigma h(\pi_G(y)))\nabla v),\mathbb{A}(y, h(y), \pa_\sigma h(y))\nabla v[\nu_G] )
$$
is of class $C^1$.
We now  check the invertibility (with continuity of the inverse) of $\pa_v \mathcal F(h, v_h)$, which is a linear operator from $C^{k,\alpha}(\Om\setminus G)$ to $C^{k-2,\alpha}(\Om\setminus G)\times C^{k-1,\alpha}(\pa G)$. This amounts to showing that for every $f\in C^{k-2,\alpha}(\Om\setminus G)$ and $g\in C^{k-1,\alpha}(\pa G)$ the system
$$
\begin{cases}
\Div(\mathbb{A}(y, h(\pi_G(y)), \pa_\sigma h(\pi_G(y)))\nabla w)=f & \text{in $\Omega\setminus G$}\\
\mathbb{A}(y, h(y), \pa_\sigma h(y))\nabla w[\nu_G]=g & \text{on $\pa G$,}\\
w=0 & \text{on $\pa_D\Om$,}
\end{cases}
$$
admits a unique solution $w\in C^{k,\alpha}(\Om\setminus G)$ such that $\|w\|_{C^{k,\alpha}(\Om\setminus G)}\leq C(\|f\|_{C^{k-2,\alpha}(\Om\setminus G)}+\|g\|_{C^{k-1,\alpha}(\pa G)})$, with $C$ depending only on  $k$, $\alpha$, $G$ and $\Om$. This follows from the classical Schauder estimates for linear elliptic systems (see for instance \cite{ADN}). Thus, since $\mathcal F(h, v_h)=0$, we may apply the Implicit Function Theorem  (see \cite[Theorem~2.3]{AP}) to deduce that there exists $\de>0$ such that the map $h\mapsto\mathcal S(h):=v_h$ is of class $C^1$ from a $\de$-neighborhood of $h$ in the $C^{k,\alpha}(\pa G)$-norm to $C^{k,\alpha}(\Om\setminus G)$. 

Since $\mathcal S$ is $C^1$ in $\mathcal U$ and $\mathfrak{h}^{k,\beta}_M(G)\subset\mathcal U$ is compact in $C^{k\,\alpha}(\pa G)$, the Fr\'echet derivative $D\mathcal S(h)$ is equibounded in $\mathcal L(C^{k,\alpha}(\pa G);C^{k,\alpha}(\Omega\setminus G))$ for $h\in\mathfrak{h}^{k,\beta}_M(G)$. Hence \eqref{linear} easily follows.
\end{proof}

Let us consider the flow  \eqref{eqf},  where $f$  satisfies the assumptions of Lemma \ref{flow exists 1}. For every given time $t$ we 
consider the elastic equilibrium $u_t$ defined in \eqref{ut}.  We start by observing that  arguing as in  \cite[Theorem 3.2]{FM09} one can show that  $\dot u_t$ satisfies 
\begin{equation}\label{eq u dot}
\int_{\Om \setminus F_t} \C E(\dot u_t) : E(\vphi) \, dx =- \int_{\pa F_t}\text{div}_\tau (\pa_{\sigma \sigma}R_t\, \C E(u_t) ) \cdot \vphi \, d \Ha^1 
\end{equation}
for all $\varphi \in H^1(\Omega \setminus F_t; \R^2)$ such that $\varphi = 0$ on $\pa_D\Omega$. Note also that $\dot u_t=0$ on $\pa_D\Omega$.

We can now prove the regularity for the elastic equilibria $u_{t}$. 

\begin{lemma} \label{regularity u}
Let  $F_0$ and $\alpha$ be as in  Lemma \ref{flow exists 1}. Fix   
\beq\label{emmezero}
M_0 > 2\|Q(E(u_G))\|_{C^{1,\alpha}(\partial G)}.
\eeq
There exist  $\delta_1\in (0,\bar \de)$,  $T_0>0$ and $\delta_0>0$ such that if $f$ satisfies \eqref{flex0},     $\|h_0\|_{H^3(\pa G)} \leq M_0$ and  $\|h_0\|_{L^{2}(\pa G)} < \delta_0$,   then the flow \eqref{eqf} exists on $(0,T_0)$, \eqref{flex1}  holds true and     for every $t \in (0,T_0)$ 
\begin{equation}\label{regu esti}
\|Q(E(u_t))\circ \pi_{F_t}^{-1}\|_{C^{1,\alpha}(\pa G)} < M_0 \quad   \text{and} \quad \int_0^{T_0} \| \pa_t \big(Q(E(u_t)) \circ \pi_{F_t}^{-1}\big)   \|_{H^{-\frac12}(\pa G)}^2 \, dt< M_0.
\end{equation}
\end{lemma}

\begin{proof}
Let $u_G$ be the elastic equilibrium in $G$.  We first recall that if $\de_0$ is as in Lemma~\ref{flow exists 1}, then by the first inequality in \eqref{flex1}  and \eqref{linear} we have
\begin{equation}\label{regu esti 2}
\|u_t\circ \pi_{F_t}^{-1} -u_G\|_{C^{2,\alpha}(\partial G)} \leq C \|h_t\|_{C^{2,\alpha}(\pa G)} \leq  C \delta_1. 
\end{equation}
Therefore, choosing   $\delta_1\in (0, \bar\de)$  sufficiently small (depending on $M_0$) and recalling \eqref{emmezero}, the first estimate in \eqref{regu esti} follows. 

For the second estimate we calculate
\[
\begin{split}
\pa_t  &\big(Q(E(u_t))(x + h_t(x)\nu_G(x))\big) \\
&= \C E(u_t) \circ \pi_{F_t}^{-1}  : \big((\nabla E(u_t) \circ \pi_{F_t}^{-1}) [\dot h_t  \nu_G]\big) + (\C E(u_t) :E(\dot u_t)) \circ \pi_{F_t}^{-1}.
\end{split}
\]
Therefore by the $C^{2,\alpha}$-bound \eqref{flex1}  on  $h_t$  and by \eqref{regu esti 2}  we have that
\beq \label{regu est 3}
\|\pa_t  \big( Q(E(u_t)) \circ \pi^{-1}_{F_t}\big) \|_{H^{-\frac12}(\pa G)} \leq C(M_0) \|\dot h_t\|_{L^{2}(\pa G)} + C(M_0) \|\C E(u_t) :E(\dot u_t)\|_{H^{-\frac12}(\pa F_t)}.
\eeq 

Observe first that the normal velocity of $\pa F_t$  can be written on $\pa G$ as 
\[
V \circ \pi_{F_t}^{-1}  = \dot h_t \frac{1+h_tk_G}{J_t}.
\]
Therefore by the definition of the flow, recalling  the interpolation inequality \eqref{inter2}, we get
\beq \label{regu est 4}
\|\dot h_t\|_{L^2(\pa G)} \leq C \|V\|_{L^2(\pa F_t)} =  C\|\pa_{\sigma\sigma}R_t\|_{L^{2}(\pa F_t)} \leq C\|\pa_{\sigma\sigma\sigma}R_t\|_{L^2(\pa F_t)}^{\frac12}\|\pa_{\sigma}R_t\|_{L^2(\pa F_t)}^{\frac12}.
\eeq 

In order to estimate second term in \eqref{regu est 3} we recall that $\C E(u_t)[\nu] = 0$. Therefore, using  \eqref{regu esti 2} again, we get
\beq \label{est udot 1}
\begin{split}
 \|\C E(u_t) :E(\dot u_t)\|_{H^{-\frac12}(\pa F_t)} &= \|\C E(u_t) : \nabla\dot u_t\|_{H^{-\frac12}(\pa F_t)} \\
&=  \|\C E(u_t) : \nabla_\tau\dot u_t\|_{H^{-\frac12}(\pa F_t)} \leq C \|\nabla_\tau \dot u_t\|_{H^{-\frac12}(\pa F_t)}.
\end{split}
\eeq
Choosing as a  test function $\vphi =  \dot u_t$  in the equation \eqref{eq u dot} we get arguing as above   that 
\beq \label{est udot 2}
\begin{split}
2 \int_{F_t} Q(E(\dot u_t)) \, dx &= - \int_{\pa F_t}\text{div}_\tau (\pa_{\sigma\sigma}R_t \,\C E(u_t) ) \cdot \dot u_t \, d \Ha^1\\
& =    \int_{\pa F_t} \pa_{\sigma\sigma}R_t \,\C E(u_t) : \nabla_\tau\dot u_t \, d \Ha^1\\
&\leq C(M_0) \|\pa_{\sigma\sigma}R_t \|_{H^{\frac12}(\pa F_t)}\|\nabla_\tau\dot u_t\|_{H^{-\frac12}(\pa F_t)}.
\end{split}
\eeq
Recall that $\dot u_t = 0$ on $\pa_D\Omega$. Therefore by  Korn's inequality, by \eqref{est udot 2} and by the interpolation inequality \eqref{inter1} we get 
\[
\begin{split}
\|\nabla_\tau\dot  u_t\|_{H^{-\frac12}(\pa F_t)}^2 &\leq C \|\dot u_t\|_{H^{\frac12}(\pa F_t)}^2\leq C\int_{\Omega\setminus F_t} |\nabla \dot u_t|^2 \, dx \\
&\leq C\int_{F_t}  Q(E(\dot u_t)) \, dx \leq C(M_0) \|\pa_{\sigma\sigma}R_t \|_{H^{\frac12}(\pa F_t)}\|\nabla_\tau\dot u_t\|_{H^{-\frac12}(\pa F_t)}\\
&\leq C(M_0) \big(\|\pa_{\sigma \sigma \sigma}R_t\|_{L^2}^{\frac34}\|\pa_\sigma R_t\|_{L^2}^{\frac14}+\|\pa_\sigma R_t\|_{L^2}\big)\|\nabla_\tau\dot u_t\|_{H^{-\frac12}(\pa F_t)}.
\end{split}
\]
Note that the first inequality above  uses  the fact that the tangential derivative is a continuous operator from $H^{1/2}$ to $H^{-1/2}$. This is a well-known fact, see for instance \cite[Theorem~8.6]{FM09}. 
This inequality together with \eqref{est udot 1} yields
\beq \label{est udot 3}
 \|\C E(u_t) :E(\dot u_t)\|_{H^{-\frac12}(\pa F_t)} \leq C(M_0) \big(\|\pa_{\sigma \sigma \sigma}R_t\|_{L^2(\pa F_t)}^{\frac34}\|\pa_\sigma R_t\|_{L^2(\pa F_t)}^{\frac14}+\|\pa_\sigma R_t\|_{L^2(\pa F_t)}\big).
 \eeq
 
By combining the estimates  \eqref{regu est 3}, \eqref{regu est 4} and  \eqref{est udot 3} we deduce  that for every $\eps>0$ 
\[
\int_0^{T_0}\| \big(\pa_t  Q(E(u_t))\big) \circ \pi^{-1}_{F_t} \|_{H^{-\frac12}(\pa G)}^2\,dt \leq \int_0^{T_0} \big(\eps  \|\pa_{\sigma \sigma \sigma}R_t\|_{L^2(\pa F_t)}^2 + C_\eps\|\pa_\sigma R_t\|_{L^2(\pa F_t)}^2\big)\,dt. 
\]
Hence, taking $\eps$ sufficiently small we obtain  by \eqref{stimaf2} and by Lemma \ref{flow exists 1} 
\[
\begin{split}
\int_0^{T_0} \| \big(\pa_t  &Q(E(u_t)) \big) \circ \pi^{-1}_{F_t} \|_{H^{-\frac12}(\pa G)}^2\,dt \leq \\
&\leq \frac12 \int_0^{T_0}\| \dot f_t  \|_{H^{-\frac12}(\pa G)}^2\,dt  + C(M_0) \sup_{(0,T_0)} \big( 1+  \|\pa_\sigma R_t\|_{L^2(\pa F_t}^2 )^q T_0\leq \frac12 M_0 + C T_0.
\end{split}
\]
Thus the second estimate in \eqref{regu esti} follows by choosing $T_0$ sufficiently small. 
\end{proof}

\smallskip

\begin{proof}[\textbf{Proof of Theorem~\ref{th:existence}}]
We divide the proof into several steps.

\noindent{\bf Step 1.}  Fix $\mu\in (0,1)$ and let $M_0$, $T_0$, $\alpha$, $\delta_1$ and $\delta_0$ be as in 
Lemma~\ref{regularity u}.   Let $f_1$, $f_2\in C^\infty(\pa G\times [0,+\infty)$ satisfy the assumptions of 
Lemma~\ref{regularity u}, let $h_0\in C^\infty(\pa G)$ satisfy  $\|h_0\|_{H^3(\pa G)}\leq M_0$, $\|h_0\|_{L^2(\pa G)} < \delta_0$, and let  $F_{t,i} $ be a solution of \eqref{eqf} with $f$ replaced by $f_i$. 
Denote by $h_{t,i}$ the function such that 
$\pa F_{t,i}= \{ x + h_{t,i}(x) \nu_G(x) : x \in \pa G\}$.
 We start by showing that there exists $T\in (0, T_0)$ such that    
\beq 
\label{contraction}
\int_0^{T} \int_{\pa G} (h_{t,2}  - h_{t,1})^2 \, d\Ha^1dt \leq \mu \int_0^{T} \int_{\pa G} (f_2 - f_1)^2 \, d\Ha^1dt.
\eeq
Note in particular that the above inequality implies the uniqueness of the solution of \eqref{eqf} when all the data are smooth.  

Recall that by Lemma \ref{flow exists 1} we have that $\|h_{t,i}\|_{C^{2,\alpha}} \leq \delta_1$ for all $t \in (0,T_0)$, for  $i = 1,2$.
Note that we may write the equation \eqref{eqf}  as 
\beq \label{eq on G}
\frac{(1+ h_{t,i} k_G)}{J_{t,i}} \, \dot h_{t,i} =  \frac{1}{J_{t,i}} \pa_\sigma \left(\frac{1}{J_{t,i}} \pa_\sigma  \left( (g(\nu_{F_{t,i}}) k_{F_{t,i}}) \circ \pi_{F_{t,i}}^{-1} + f_i \right) \right)\quad\text{ on }\pa G
\eeq
for $i = 1,2$ respectively, where $J_{t,i} = \sqrt{(1 + h_{t,i} k_G)^2 + (\pa_\sigma h_{t,i})^2}$. 
To simplify the notation we write
$g_{t, i}$ and $k_{t, i}$ in place of  $g(\nu_{F_{t,i}})$ and  $k_{F_{t,i}}$, respectively. Note that by the 
$C^{2,\alpha}$ bounds on $h_{t,2}$ and $h_{t,1}$, and by the expressions \eqref{formula tangent} and \eqref{formula normal}  we may estimate 
\beq \label{uniso lip}
|g_{t,2}(\pi_{F_{t,2}}^{-1}(x)) - g_{t,1}(\pi_{F_{t,1}}^{-1}(x))| 
\leq C \left(|\pa_\sigma (h_{t,2}-h_{t,1})(x)| + |(h_{t,2}-h_{t,1})(x)| \right)
\eeq
for every $x \in \pa G$. Moreover from the  expression  \eqref{curvature formula}, from the  $C^{2,\alpha}$-bounds on $h_2$ and $h_1$, and from the ellipticity 
assumptions on $\vphi$  we deduce that there are positive constants $c$ and $C$ such that 
\beq \label{curv lip}
\begin{split}
[k_{t,2}(\pi_{F_{t,2}}^{-1}) - k_{t,1}(\pi_{F_{t,1}}^{-1})]& \, \pa_{\sigma \sigma}(h_{t,2}-h_{t,1}) \\
&\leq - c |\pa_{\sigma \sigma}(h_{t,2}- h_{t,1})|^2 + C \left(|\pa_\sigma (h_{t,2}-h_{t,1})|^2 + |(h_{t,2}-h_{t,1})|^2 \right)
\end{split}
\eeq
on $\pa G$.

Multiply the equation \eqref{eq on G} by $h_{t,2}-h_{t,1}$ for $i = 1,2$, integrate over $\pa G$ and integrate by parts twice to get

\[
\begin{split}
\int_{\pa G} \dot h_{t,i} & \, (h_{t,2} - h_{t,1}) \,d \Ha^1 \\
& = \int_{\pa G}  \pa_\sigma \left(\frac{1}{J_{t,i}} \pa_\sigma  \left( (g_{t,i} \,  k_{t,i})\circ \pi_{F_{t,i}}^{-1}  + f_i \right) \right) \, \left( \frac{1}{1 + h_{t,i} k_G}(h_{t,2} - h_{t,1}) \right) \, d \Ha^1\\
&= \int_{\pa G}  \left( (g_{t,i} \,  k_{t,i} )\circ \pi_{F_{t,i}}^{-1}  + f_i  \right) \,  \pa_\sigma \left(\frac{1}{J_{t,i}}  \pa_\sigma  \left( \frac{1}{1 + h_{t,i} k_G} (h_{t,2} - h_{t,1}) \right) \right)\, d \Ha^1.
\end{split}
\]
Substract the equation for $i = 1$ from the equation for $i = 2$ to get 
\[
\begin{split}
\frac{d}{dt}  \biggl(\frac12\int_{\pa G}  & \, (h_{t,2} - h_{t,1})^2 \,d \Ha^1 \biggr)   = \int_{\pa G} (h_{t,2} - h_{t,1}) \, (\dot h_{t,2} - \dot h_{t,1}) \,d \Ha^1 \\
&=\int_{\pa G}   ( (g_{t,2} \,  k_{t,2} )\circ \pi_{F_{t,2}}^{-1}  + f_2 )  \,  \pa_\sigma \left(\frac{1}{J_{t,2}}  \pa_\sigma  \left( \frac{1}{1 + h_{t,2} k_G} (h_{t,2} - h_{t,1}) \right) \right) \, d \Ha^1\\
&\,\,\,\,\, -\int_{\pa G}   ( (g_{t,1} \,  k_{t,1} )\circ \pi_{F_{t,1}}^{-1}   + f_1 )  \,  \pa_\sigma \left(\frac{1}{J_{t,1}}  \pa_\sigma  \left( \frac{1}{1 + h_{t,1} k_G} (h_{t,2} - h_{t,1}) \right) \right) \, d \Ha^1.
\end{split}
\]
By the $C^{2,\alpha}$-bounds  on $h_{t,2}, h_{t,1}$, by $C^{0,\alpha}$-bounds  on $f_2, f_1$,  by \eqref{curvature formula} and by \eqref{uniso lip} and  \eqref{curv lip}   we  conclude that there are positive constants $c$ and $C$ such that 
\beq \label{exist pt1}
\begin{split}
 \frac{d}{dt} \biggl(\frac12\int_{\pa G}   \, (h_{t,2} & - h_{t,1})^2 \,d \Ha^1 \biggr)  + c \int_{\pa G} |\pa_{\sigma \sigma}(h_{t,2} -h_{t,1})|^2\,d \Ha^1  \\
 &\leq  C \int_{\pa G} (f_2 -f_1)^2\,d \Ha^1 + C \int_{\pa G} |\pa_{\sigma}(h_{t,2} -h_{t,1})|^2 + (h_{t,2}-h_{t,1})^2 \,d \Ha^1. 
\end{split}
\eeq

Denote $w_t(x) := h_{t,2}(x)-h_{t,1}(x)$. By interpolation we have
\[
\|\pa_\sigma w_t\|_{L^2}^2 \leq C\, \|\pa_{\sigma \sigma} w_t\|_{L^2} \|w_t\|_{L^2} + C \|w_t\|_{L^2}^2 \leq \eps \|\pa_{\sigma \sigma} w_t\|_{L^2}^2 + C_\eps \|w_t\|_{L^2}^2.
\]
Hence, for $\eps$  small enough, we obtain by \eqref{exist pt1} that 
\beq \label{exist pt2}
 \frac{d}{dt} \left(\int_{\pa G}   \, w_t^2 \, \Ha^1 \right) \leq C \int_{\pa G} w_t^2\,d \Ha^1  + C \int_{\pa G} (f_2 -f_1)^2\,d \Ha^1.
\eeq
Take $T \in (0,T_0)$. Recall that $w_0\equiv 0$. Therefore integrating \eqref{exist pt2} over $(0,t)$, with  $t\in (0, T)$, implies
\beq \label{exist pt3}
  \int_{\pa G}   \, w_t^2 \,d \Ha^1  \leq C  \int_0^{T} \int_{\pa G}w_s^2\,d \Ha^1 ds + C  \int_0^{T} \int_{\pa G} (f_2 -f_1)^2\,d \Ha^1 ds.
\eeq
Integrating \eqref{exist pt3} over $(0,T)$ yields
\[
\int_0^{T} \int_{\pa G}   \, w_t^2 \,d \Ha^1 dt  \leq CT   \int_0^{T} \int_{\pa G}w_t^2\,d \Ha^1 dt  + CT  \int_0^{T} \int_{\pa G} (f_2 -f_1)^2\,d \Ha^1 dt.
\]
Therefore  \eqref{contraction} follows by taking   $T\in (0, T_0)$  sufficiently small. 

\medskip

\noindent {\bf Step 2.}   Fix $M_0> 2 
\|Q(E(u_G))\|_{C^{1,\alpha}(\partial G)}$,  $\mu \in (0,1)$, and let $T$, $\de_0$ be as in Step 1.
Define the function space
\[
\mathcal{C}:= \Big{\{} f \in L^2(0,T; L^2(\pa G)) \, : \, \sup_{(0,T)} \|f\|_{C^{1,\alpha}} \leq M_0, \, \int_0^{T} \| \dot f_t   \|_{H^{-\frac12}(\pa G)}^2 \leq M_0 \Big{\}}.
\]
We want to show that for every $f\in \mathcal{C}$ equation \eqref{eqf} has a solution in the interval $(0,T)$, and that \eqref{contraction} holds for $f_1$, $f_2\in \mathcal{C}$.
 
 To this end, fix $h_0\in H^3(\pa G)$ satisfying  $\|h_0\|_{H^3(\pa G)}\leq M_0$, $\|h_0\|_{L^2(\pa G)} < \delta_0$, and let $f\in \mathcal{C}$. Consider a sequence $f_n\in \mathcal{C}\cap C^\infty(\pa G\times [0,+\infty))$ such that  $f_n\to f$ in $L^2(0, T; L^2(\pa G))$ and a sequence of smooth functions $h_n$ such that 
 $\|h_n\|_{H^3(\pa G)}\leq M_0$ and $h_n \to h_0$ in $L^2(\pa G)$. Denote by $F_{t, n}$ the solution
 of  \eqref{eqf} with forcing term $f_n$ and initial datum $h_n$, and let $h_{t,n}$ be the function on $\pa G$  such that $\pa F_{t,n}= \{ x + h_{t,n}(x) \nu_G(x) : x \in \pa G\}$.

Observe that from \eqref{flex1} we have that 
$$
\sup_n\sup_{(0,T)}\bigl(\|h_{t, n}\|_{C^{2, \alpha}(\pa G)} + \|\pa_\sigma R_{t,n}\|_{L^{2}(\pa F_t)}\bigr)<+\infty,
$$
where $R_{t,n}$ is defined as in \eqref{Rt} with $f$ replaced by $f_n$. In turn \eqref{stimaf2} yields that 
$R_{t,n}$ is uniformly bounded in $L^2(0, T; H^3(\pa G))$ and thus $h_{t,n}$ is uniformly bounded in 
$H^1(0, T; H^1(\pa G))$. Therefore, up to a (not relabelled) subsequence, we may assume that 
$h_{t,n}\wto h_t$ weakly in $H^1(0, T; H^1(\pa G))$ and, recalling the uniform $C^{2,\alpha}$ bounds on 
$h_{t,n}$ we may conclude that in fact  $h_{t,n}\to h_t$ in $C^{2,\beta}(\pa G)$ for all $\beta\in (0, \alpha)$ and for all $t\in (0, T)$ and thus $R_{t,n}\circ \pi_{F_{t,n}}^{-1}\to R_t\circ \pi_{F_{t}}^{-1}$ in $C^{0, \beta}(\pa G)$, where $F_t$ is the set corresponding to $h_t$. It is now easy to see that the equation passes to the limit and $F_t$ is a solution of  \eqref{eqf} with initial datum $h_0$  and forcing term $f$.  Note also that the same approximation argument yields that \eqref{contraction} holds true also in the case where $f_1$, $f_2\in \mathcal C$, so that in particular the solution is unique also in this case. Moreover, again by approximation,  the conclusions of Lemmas~\ref{flow exists 1} and \ref{regularity u} remain true. 

\medskip

\noindent{\bf Step 3.} Fix $h_0$ as in Step 2 and consider the  map $\mathcal T : \mathcal{C} \to \mathcal{C}$ defined by $Tf(\cdot, t) =- Q(E(u_t))\circ \pi_{F_t}^{-1}$ for all $t\in [0, T)$, where $F_t$ is the solution of  \eqref{eqf} with initial datum $h_0$ and forcing term $f$ (and $u_t$ is the elastic equilibrium in $\Omega\setminus F_t$). From  \eqref{regu esti}, which holds also in our case thanks to the previous step,  it follows  that the map is well defined. In order to conclude the proof it is enough to show that $\mathcal T$ is a contraction and thus it admits a fixed point.
To this aim, with the same notation of Step 1, for any $f_1$, $f_2\in \mathcal C$ and for  any $\e>0$ we have
\begin{align*}
\int_0^{T}&\|Q(E(u_{t,1}))\circ \pi_{F_{t,1}}^{-1}-Q(E(u_{t,2}))\circ \pi_{F_{t,2}}^{-1}\|^2_{L^2(\pa G)}\, dt
 \\
 &\leq C\int_0^{T}\|Q(E(u_{t,1}))\circ \pi_{F_{t,1}}^{-1}-Q(E(u_{t,2}))\circ \pi_{F_{t,2}}^{-1}\|^2_{C^{0,\alpha}(\pa G)}\, dt\\
 & \leq C \int_0^{T}\|h_{t, 1}-h_{t, 2}\|_{C^{1, \alpha}(\pa G)}^2\, dt \\
 &\leq C  \int_0^{T}\bigl[\|\pa_{\sigma\sigma}(h_{t, 1}-h_{t, 2})\|_{L^2}^{2\theta}\|h_{t, 1}-h_{t, 2}\|_{L^2}^{2(1-\theta)}+\|h_{t, 1}-h_{t, 2}\|_{L^2}^{2}\bigr]\, dt\\
 &\leq \int_0^{T}\bigl[\e\|\pa_{\sigma\sigma}(h_{t, 1}-h_{t, 2})\|_{L^2}^{2}+C_\e\|h_{t, 1}-h_{t, 2}\|_{L^2}^{2}\bigr]\, dt,
\end{align*}
where we used \eqref{linear} and \eqref{inter3}.
We use  \eqref{exist pt1} and \eqref{inter2}, argue as in Step 1, to control the last integral in the above chain of inequalities and deduce that there exists $C_1>0$ independent of $\e$ such that 
\begin{align*}
\int_0^{T}\|Q(E(u_{t,1}))\circ \pi_{F_{t,1}}^{-1}&-Q(E(u_{t,2}))\circ \pi_{F_{t,2}}^{-1}\|^2_{L^2(\pa G)}\, dt\\
&\leq C_1 \e \int_0^{T}\|f_1-f_2\|_{L^2}^2\, dt+ C_\e\int_0^{T}\|h_{t, 1}-h_{t, 2}\|_{L^2}^{2}\, dt\\
&\leq C_1 \e \int_0^{T}\|f_1-f_2\|_{L^2}^2\, dt+ C_\e\mu \int_0^{T}\|f_1-f_2\|_{L^2}^2\, dt,
\end{align*}
where the last inequality follows from \eqref{contraction}. The conclusion follows by taking $\e$ and then $\mu$ sufficiently small. 
\end{proof}

We conclude this section by showing that the solution provided by Theorem~\ref{th:existence} is in fact classical, namely of class 
$C^{\infty}$.
\begin{theorem}\label{Cinfty}
 Under the assumptions of  Theorem~\ref{th:existence} we have $(h_t)_{t\in (0,T)}\in C^\infty(0,T;C^{\infty}(\pa G))$.
\end{theorem}
\begin{proof}
As in the proof of Theorem~\ref{th:existence} we rewrite the equation on $\pa G$, see \eqref{eq on G}, thus getting
\beq \label{Cinfty1}
\frac{(1+ h_{t} k_G)}{J_{t}} \, \dot h_{t} =  \frac{1}{J_{t}} \pa_\sigma \left(\frac{1}{J_{t}} \pa_\sigma  \left( (g(\nu_{F_{t}}) k_{F_{t}}) \circ \pi_{F_{t}}^{-1} - Q_t \right) \right)\quad\text{ on }\pa G,
\eeq
where we have set $Q_t:=Q(E(u_t))\circ\pi_{F_{t}}^{-1}$. We  divide the proof in four steps.

\medskip

\noindent{\bf Step 1.}
Given $\Delta t\not=0$, let us subtract the equation above at time $t$ from the same equation at time $t+\Delta t$ and multiply both sides by $h_{t+\Delta t}-h_t$. Then, integrating by part and arguing as in the proof of \eqref{exist pt1} we get, using also Proposition~\ref{interpolation} to estimate $\|\pa_{\sigma}(h_{t+\Delta t}-h_t)\|_{L^2(\pa G)}$,
\beq \label{Cinfty2}
\begin{split}
 \frac{d}{dt} \biggl(\frac12\int_{\pa G}   \, (h_{t+\Delta t}  - h_{t})^2 \,d \Ha^1 \biggr) & +  c \int_{\pa G} |\pa_{\sigma \sigma}(h_{t+\Delta t} -h_{t})|^2\,d \Ha^1  \\
 &\leq  C \int_{\pa G} (Q_{t+\Delta t} -Q_t)^2\,d \Ha^1 + C\int_{\pa G}(h_{t+\Delta t}-h_{t})^2 \,d \Ha^1. 
\end{split}
\eeq
Fix now $\alpha\in (0,\frac12)$. Using Proposition~\ref{interpolation}  and the estimate \eqref{linear} with $F$ and $G$ replaced respectively by $F_{t+\Delta t}$ and $F_t$ we have
\begin{align*}
\|Q_{t+\Delta t} -Q_t\|_{L^2(\pa G)}
 &\leq C\|Q_{t+\Delta t} -Q_t\|_{C^{0,\alpha}(\pa G)}\leq C\|h_{t+\Delta t} -h_t\|_{C^{1,\alpha}(\pa G)} \\
& \leq C\big(\|\pa_{\sigma\sigma}(h_{t+\Delta t} -h_t)\|_{L^2(\pa G)}^{\vartheta}\|h_{t+\Delta t} -h_t\|_{L^2(\pa G)}^{1-\vartheta}+\|h_{t+\Delta t} -h_t\|_{L^2(\pa G)}\big).
\end{align*}
Inserting this inequality in \eqref{Cinfty2} we get 
\[
 \frac{d}{dt} \biggl(\frac12\int_{\pa G}   \, (h_{t+\Delta t}  - h_{t})^2 \,d \Ha^1 \biggr)  + c \int_{\pa G} |\pa_{\sigma \sigma}(h_{t+\Delta t} -h_{t})|^2\,d \Ha^1  
 \leq   C\|h_{t+\Delta t} -h_t\|_{L^2(\pa G)}^2. 
\]
Then for $\mathcal L^1$-a.e. $t_0,t_1$ with $0<t_0<t_1<T$, integrating the above inequality in $(t_0,t_1)$, we have
\[
\begin{split}
\|h_{t_1+\Delta t}  - h_{t_1} \|_{L^2(\pa G)}^2  +\int_{t_0}^{t_1}&\|\pa_{\sigma \sigma}(h_{t+\Delta t} -h_{t})\|_{L^2(\pa G)}^2\,dt \\
&\leq \|h_{t_0+\Delta t}  - h_{t_0} \|_{L^2(\pa G)}^2+C\int_{t_0}^{t_1}\|h_{t+\Delta t} -h_{t}\|_{L^2(\pa G)}^2\,dt.
\end{split}
\]
Finally, dividing both sides of this inequality by $(\Delta t)^2$, letting $\Delta t\to0$ and recalling that $h\in H^1(0,T;H^1(\pa G))$, we  conclude that for any time interval $J\subset\!\subset(0,T)$
\beq\label{Cinfty3}
\sup_{t\in J}\| \dot h_t\|_{L^2(\pa G)}^2+\int_J\|\pa_t(\pa_{\sigma \sigma}h_t)\|_{L^2(\pa G)}^2\,dt<\infty.
\eeq

\medskip

\noindent{\bf Step 2.} We start again by subtracting equation \eqref{Cinfty1} at time $t$ from the same equation at time $t+\Delta t$. We now multiply  both sides by $\pa_{\sigma\sigma}(h_{t+\Delta t}-h_t)$. Then, arguing as in the proof of \eqref{Cinfty2} we get  the following estimate
\beq \label{Cinfty4}
\begin{split}
 \frac{d}{dt} \biggl(\frac12\int_{\pa G} (\pa_\sigma(h_{t+\Delta t}  - &h_{t}))^2 \,d \Ha^1 \biggr)  +  c \int_{\pa G} |\pa_{\sigma \sigma\sigma}(h_{t+\Delta t} -h_{t})|^2\,d \Ha^1  \\
 &\leq  C \int_{\pa G} (\pa_\sigma(Q_{t+\Delta t} -Q_t))^2\,d \Ha^1 + C\int_{\pa G}(h_{t+\Delta t}-h_{t})^2 \,d \Ha^1. 
\end{split}
\eeq
As in the previous step we may estimate, using  \eqref{linear}
and Proposition~\ref{interpolation},  
\begin{align*}
\|\pa_\sigma(Q_{t+\Delta t} -Q_t)\|_{L^2(\pa G)}
& \leq
C\|\pa_\sigma(Q_{t+\Delta t} -Q_t)\|_{C^{0,\alpha}(\pa G)}\leq C\|h_{t+\Delta t} -h_t\|_{C^{2,\alpha}(\pa G)} \\
&\leq C\|\pa_{\sigma\sigma\sigma}(h_{t+\Delta t} -h_t)\|_{L^2}^{\vartheta}\|h_{t+\Delta t} -h_t\|_{L^2}^{1-\vartheta}+C\|h_{t+\Delta t} -h_t\|_{L^2}.
\end{align*}
Using this estimate and integrating \eqref{Cinfty4} in $(t_0,t_1)$ for $\mathcal L^1$-a.e. $t_0,t_1$ with $0<t_0<t_1<T$, we have
\[
\begin{split}
\|\pa_\sigma(h_{t_1+\Delta t} - h_{t_1})\|_{L^2(\pa G)}^2  &+c \int_{t_0}^{t_1}\|\pa_{\sigma \sigma\sigma}(h_{t+\Delta t} -h_{t})\|_{L^2(\pa G)}^2\,dt \\
&\leq \|\pa_\sigma(h_{t_0+\Delta t}  - h_{t_0}) \|_{L^2(\pa G)}^2+ C\int_{t_0}^{t_1}\|h_{t+\Delta t} -h_{t}\|_{L^2(\pa G)}^2\,dt 
\end{split}
\]
Divide both sides of this inequality by $(\Delta t)^2$ and recall that 
$\pa_{\sigma\sigma\sigma}h_t\in L^2(0,T;L^2(\pa G))$ and that by \eqref{Cinfty3} $\pa_t(\pa_\sigma h_t)\in L^2_{loc}(0,T;H^1(\pa G))$. Using this information and letting $\Delta t\to0$ we conclude that for every interval $J\subset\!\subset(0,T)$
\beq \label{Cinfty5}
\sup_{t\in J}\| \pa_t(\pa_\sigma h_t)\|_{L^2(\pa G)}^2+\int_J\|\pa_t(\pa_{\sigma \sigma\sigma}h_t)\|_{L^2(\pa G)}^2\,dt<\infty.
\eeq
Note that from the previous inequality and by the equation \eqref{eq on G} we have that for every interval $J\subset\!\subset(0,T)$
$$
\sup_{t\in J}\big(\|\pa_{\sigma\sigma\sigma}R_t\|_{L^2(\pa G)}+\|\pa_{\sigma\sigma\sigma}h_t\|_{L^2(\pa G)}\big)<\infty.
$$
In particular, from this inequality we deduce that
$$
\sup_{t\in J}\big(\|\pa_{\sigma}R_t\|_{C^{0,\alpha}(\pa G)}+\|h_t\|_{C^{2,\alpha}(\pa G)}\big)<\infty.
$$
In turn, since $\|\pa_\sigma Q_t\|_{C^{0,\alpha}(\pa G)}\leq C(G,\|h_t\|_{C^{2,\alpha}(\pa G)})$, the above inequality implies immediately that
\beq \label{Cinfty5.1}
\sup_{t\in J}\|h_t\|_{C^{3,\alpha}(\pa G)}<\infty
\eeq

\medskip

\noindent{\bf Step 3.} At this point we would like to continue as before, subtracting the equations \eqref{Cinfty1} at times $t+\Delta t$ an $t$ and multiplying the resulting difference by $\pa_{\sigma\sigma\sigma\sigma}(h_{t+\Delta t}-h_t)$. 
However this argument only works provided we know that $h\in L^2_{loc}(0,T;H^4(\pa G))$. 

To prove this property of $h$ we go back to equation \eqref{Cinfty1} and, denoting by $s$ the arclength on $\pa G$, we subtract the equation for $h$ from the same equation  for $h(\cdot+\Delta s)$, where $\Delta s$ is a non zero increment of the arclength. Then we multiply both sides by $\pa_{\sigma\sigma}h(\cdot+\Delta s)-\pa_{\sigma\sigma} h$ to deduce with the usual calculations that
\[
\begin{split}
 \frac{d}{dt} \biggl(\frac12\int_{\pa G} (\pa_\sigma(h_t(\cdot+& \Delta s)  - h_{t}))^2 \,d \Ha^1 \biggr)  +  c \int_{\pa G} |\pa_{\sigma \sigma\sigma}(h_t(\cdot+\Delta s) -h_{t})|^2\,d \Ha^1  \\
 &\leq  C \int_{\pa G} (\pa_\sigma(Q_t(\cdot+\Delta s) -Q_t))^2\,d \Ha^1 +C \int_{\pa G}(h_t(\cdot+\Delta s) -h_{t})^2 \,d \Ha^1. 
\end{split}
\]
As before we estimate
\begin{align*}
\|\pa_\sigma(Q_t(\cdot+\Delta s) &-Q_t)\|_{L^2(\pa G)}\leq C\|h_t(\cdot+\Delta s) -h_t)\|_{C^{2,\alpha}(\pa G)} \\
& \leq C\|\pa_{\sigma\sigma\sigma}(h_t(\cdot+\Delta s) -h_t)\|_{L^2}^{\vartheta}\|h_t(\cdot+\Delta s) -h_t\|_{L^2}^{1-\vartheta}+C\|h_t(\cdot+\Delta s) -h_t\|_{L^2}
\end{align*}
so to obtain that for $\mathcal L^1$-a.e. $t_0,t_1$ with $0<t_0<t_1<T$ 
\[
\begin{split}
\|\pa_\sigma(h_{t_1}(\cdot+\Delta s)  - & h_{t_1})\|_{L^2(\pa G)}^2  +\int_{t_0}^{t_1}\|\pa_{\sigma \sigma\sigma}(h_{t}(\cdot+\Delta s) -h_{t})\|_{L^2(\pa G)}^2\,dt \\
&\leq \|\pa_\sigma(h_{t_0}(\cdot+\Delta s)  - h_{t_0})\|_{L^2(\pa G)}^2+C\int_{t_0}^{t_1}\|h_{t}(\cdot+\Delta s) -h_{t}\|_{L^2(\pa G)}^2\,dt.
\end{split}
\]
Thus, we may conclude that for every interval $J\subset\!\subset(0,T)$
$$
\sup_{t\in J}\| \pa_{\sigma\sigma} h_t\|_{L^2(\pa G)}^2+\int_J\|\pa_{\sigma \sigma\sigma\sigma}h_t\|_{L^2(\pa G)}^2\,dt<\infty.
$$
We now use this estimate, together with the estimate \eqref{Cinfty5.1} obtained in the previous step, in order to show \eqref{Cinfty6.2} below.

 To this end,   we subtract equation \eqref{Cinfty1} at time $t$ from the same equation at time $t+\Delta t$ and multiply  both sides by $\pa_{\sigma\sigma\sigma\sigma}(h_{t+\Delta t}-h_t)$ and we obtain
\[
\begin{split}
 \frac{d}{dt} \biggl(\frac12\int_{\pa G} (\pa_{\sigma\sigma}(h_{t+\Delta t}  - &h_{t}))^2 \,d \Ha^1 \biggr)  +  c \int_{\pa G} |\pa_{\sigma \sigma\sigma\sigma}(h_{t+\Delta t} -h_{t})|^2\,d \Ha^1  \\
 &\leq  C \int_{\pa G} (\pa_{\sigma\sigma}(Q_{t+\Delta t} -Q_t))^2\,d \Ha^1 + C\int_{\pa G}(h_{t+\Delta t}-h_{t})^2 \,d \Ha^1. 
\end{split}
\]
Then using, \eqref{Cinfty5.1},  \eqref{linear}
and Proposition~\ref{interpolation},  we have
\begin{align*}
\|\pa_{\sigma\sigma}(Q_{t+\Delta t} -Q_t)\|_{L^2(\pa G)}
& \leq
C\|\pa_{\sigma\sigma}(Q_{t+\Delta t} -Q_t)\|_{C^{0,\alpha}(\pa G)}\leq C\|h_{t+\Delta t} -h_t\|_{C^{3,\alpha}(\pa G)} \\
&\leq C\|\pa_{\sigma\sigma\sigma}(h_{t+\Delta t} -h_t)\|_{L^2}^{\vartheta}\|h_{t+\Delta t} -h_t\|_{L^2}^{1-\vartheta}+C\|h_{t+\Delta t} -h_t\|_{L^2}.
\end{align*}
Then, arguing as in the proof of \eqref{Cinfty5}, we get
\beq \label{Cinfty6.2}
\sup_{t\in J}\| \pa_t(\pa_{\sigma\sigma} h_t)\|_{L^2(\pa G)}^2+\int_J\|\pa_t(\pa_{\sigma \sigma\sigma\sigma}h_t)\|_{L^2(\pa G)}^2\,dt<\infty.
\eeq
Then, arguing as in the proof of \eqref{Cinfty5.1} we have that
$$
\sup_{t\in J}\|h_t\|_{C^{4,\alpha}(\pa G)}<\infty
$$

At this point we proceed by induction, obtaining at each step first an increment in the space regularity and then the corresponding estimate with respect to time. More precisely, for every  interval $J\subset\!\subset(0,T)$ and every integer $k\geq2$ we first have that
$$
\sup_{t\in J}\| \pa_{\sigma}^k h_t\|_{L^2(\pa G)}^2+\int_J\|\pa_{\sigma }^{k+2}h_t\|_{L^2(\pa G)}^2\,dt<\infty
$$
Then from this we deduce that again for every interval $J\subset\!\subset(0,T)$
$$
\sup_{t\in J}\| \pa_t(\pa_{\sigma}^k h_t)\|_{L^2(\pa G)}^2+\int_J\|\pa_t(\pa_{\sigma }^{k+2}h_t)\|_{L^2(\pa G)}^2\,dt<\infty 
$$
and in turn that
$$
\sup_{t\in J}\|h_t\|_{C^{k+2,\alpha}(\pa G)}<\infty
$$
In conclusion this proves that $h\in W^{1,\infty}_{loc}(0,T;C^{\infty}(\pa G))$. 

\medskip

\noindent{\bf Step 4.} Let us now show the full regularity of $h$ with respect to time. As in Step 1 we fix $\Delta t\not=0$ and subtract equation \eqref{Cinfty1} from the same equation at time $t+\Delta t$. However, differently from before, we multiply both sides of this difference by  $\dot h_{t+\Delta t}-\dot h_t$.  Then, a simple use of Young's inequality and Proposition~\ref{interpolation} yields
\beq \label{Cinfty7}
\begin{split}
\int_{\pa G}  (\dot h_{t+\Delta t}  -\dot h_{t})^2 \,d \Ha^1 & \leq C \int_{\pa G} |\pa_{\sigma \sigma\sigma\sigma}(h_{t+\Delta t} -h_{t})|^2\,d \Ha^1  \\
 & +C \int_{\pa G} (\pa_{\sigma\sigma}(Q_{t+\Delta t} -Q_t))^2\,d \Ha^1 + C\int_{\pa G}(h_{t+\Delta t}-h_{t})^2 \,d \Ha^1. 
\end{split}
\eeq
Then we estimate as usual
\begin{align*}
\|\pa_{\sigma\sigma}(Q_{t+\Delta t} -Q_t)\|_{L^2}
& \leq
C\|\pa_{\sigma\sigma}(Q_{t+\Delta t} -Q_t)\|_{C^{0,\alpha}}\leq C\|h_{t+\Delta t} -h_t\|_{C^{3,\alpha}} \\
&\leq C\|\pa_{\sigma\sigma\sigma\sigma}(h_{t+\Delta t} -h_t)\|_{L^2}^{\vartheta}\|h_{t+\Delta t} -h_t\|_{L^2}^{1-\vartheta}+C\|h_{t+\Delta t} -h_t\|_{L^2}.
\end{align*}
Thus, from \eqref{Cinfty7} one gets
$$
\int_{\pa G}  (\dot h_{t+\Delta t}  -\dot h_{t})^2 \,d \Ha^1  \leq C \int_{\pa G} |\pa_{\sigma \sigma\sigma\sigma}(h_{t+\Delta t} -h_{t})|^2\,d \Ha^1  + C\int_{\pa G}(h_{t+\Delta t}-h_{t})^2 \,d \Ha^1. 
$$
Dividing this inequality by $(\Delta t)^2$ and recalling what was proved in Step 3 we conclude that for every interval $J\subset\!\subset(0,T)$
$$
\sup_{t\in J}\| \pa_{tt} h_t\|_{L^2(\pa G)}^2<\infty.
$$
Similarly, differentiating $k$ times the equation \eqref{Cinfty1} and arguing as before we conclude that indeed for every integer $k$ and for every interval $J\subset\!\subset(0,T)$
$$
\sup_{t\in J}\| \pa_{tt}(\pa_{\sigma}^kh_t)\|_{L^2(\pa G)}^2<\infty.
$$
Then we have that $h\in W^{2,\infty}_{loc}(0,T;C^{\infty}(\pa G))$. Finally, differentiating \eqref{Cinfty1} with respect to $t$ and repeating the same argument as before we end up by proving that $h\in W^{k,\infty}_{loc}(0,T;C^{\infty}(\pa G))$ for every integer $k\geq2$. This concludes the proof.
\end{proof}

\section{Asymptotic Stability} \label{sec:stability}

In this section we address the long-time behavior of the flow for a special class of initial data.
 
To this aim, we start by noticing  that if $G$ is stationary, then a standard bootstrap argument shows that in fact $G$ is of class 
$C^\infty$. Moreover, by the results in \cite{KLM} $G$ turns out to be analytic.  
Recall also that the definition of stationary set  is weaker than the notion of criticality, where one requires  the first variation to be    constant on the whole $\pa G$ (see Remark~\ref{rm:station}). 

However, the above definition  fits better in our framework, since  during the evolution there is no mass transfer from one Jordan component to the other. More precisely, denoting as before by $F_{t,i}$ the bounded open set enclosed by the $i$-th connected component $\Gamma_{F_t, i}$ of $\pa F_t$, one has that the area $|F_{t,i}|$ is preserved during the flow. Indeed, one has
\beq\label{componenti}
\frac{d}{ds}|F_{t+s,i}|_{|_{s=0}}=\int_{\Gamma_{F_t, i}} V_t\, d\Ha^1=
\int_{\Gamma_{F_t, i}} \pa_{\sigma\sigma}R_t\, d\Ha^1=0.
\eeq
We are now ready to state the main result of this section.

\begin{theorem}
\label{thmstability}
Let $G \subset\subset\Om$ be a regular strictly stable stationary set   in the sense of Definition~\ref{def:stable} and fix $M>0$, $\alpha\in (0,1)$. There exists $\delta_0>0$ with the following property:  Let  $F_0\in  \mathfrak{h}^{2,\alpha}_M(\pa G)$ be  such that
 \[
|F_0 \Delta G| < \delta_0, \qquad \text{and} \qquad \int_{\pa F_0} (\pa_\sigma R_0)^2 \, d\Ha^1 < \delta_0,
\]  
where $R_0:=g(\nu_{F_0})k_{F_0}-Q(E(u_{F_0}))$ on $\pa F_0$. 
Then the unique solution $(F_t)_{t>0}$ of the flow \eqref{flow} with intial datum $F_0$ is defined for all times $t>0$. 

Moreover $F_t \to F_{\infty}$  $H^3$-exponentially fast, where $F_\infty$ is the unique stationary set in $\mathfrak{h}^{2,\alpha}_{\sigma_1}(\pa G)$
(see Proposition~\ref{stationary}) such that $|F_{\infty, i}|=|F_{0,i}|$ for $i=1,\dots, m$.  In particular, if $|F_{0,i}|=|G_i|$ for $i=1, \dots, m$, then $F_t \to G$  $H^3$-exponentially fast. 

Here $(F_{\infty, i})_{i=1, \dots, m}$ and $(F_{0,i})_{i=1, \dots, m}$  denote the open sets enclosed by the connected components $(\Gamma_{F_\infty, i})_{i=1, \dots, m}$ of $\pa F_\infty$ and $(\Gamma_{F_0, i})_{i=1, \dots, m}$ of $\pa F_0$, respectively, numbered according to \eqref{numbered}.
\end{theorem}
\begin{remark}\label{rm:precise}
In the previous statement by $H^3$ exponential convergence of $F_t$ to $F_\infty$ we mean precisely the following: 
writing $\pa F_t:=\{x+\tilde h_t(x)\nu_{F_\infty}(x):x\in \pa F_\infty\}$, we have 
\[
\|\tilde h_t\|_{H^3(\pa F_\infty)} \leq C e^{-c t}.
\]
for suitable constants $C, c>0$.
\end{remark}
\par\noindent For an example of strictly stable set $G$ to which Theorem~\ref{thmstability} applies we refer to \cite{CJP}.

In order to proof the theorem, we need the following preliminary energy identities.
\begin{proposition}
\label{energia identiteetit}
Let $(F_t)_{t \in (0,T)}$ solve \eqref{flow}. Then 
we have:
$$
\frac{d}{dt}  J(F_t) = - \int_{\pa F_t} (\pa_\sigma R_t)^2 \, d\Ha^1 
$$
and
$$
\frac{d}{dt} \left(\frac{1}{2} \int_{\pa F_t} (\pa_\sigma R_t)^2 \, d\Ha^1 \right) = - \pa^2 J[\pa_{\sigma\sigma}R_t] - \frac12 \int_{\pa F_t} k_t \, (\pa_\sigma R_t)^2 \, \pa_{\sigma \sigma}R_t\, d \Ha^1.  
$$
\end{proposition}

\begin{proof}
The first identity  follows immediately recalling  that $R_t = g(\nu_t)k_t - Q(E(u_t))$ is the first variation of the energy at $J(F_t)$, and thus  
\[
\frac{d}{dt} J(F_t) = \int_{\pa F_t} R_t (X_t \cdot \nu_t) \, d \Ha^1 =  \int_{\pa F_t} R_t\, \pa_{\sigma \sigma}R_t\, d \Ha^1 = -  \int_{\pa F_t} (\pa_{\sigma}R_t)^2 \, d \Ha^1.
\]

For the second identity we note that  the calculations leading to \eqref{app6} still apply  with $f_t \circ \pi$ replaced by $-Q(E(u_t))$ on $\pa F_t$.  Hence we have
\beq
\label{app7}
\begin{split}
\frac{d}{dt} \left(\frac{1}{2} \int_{\pa F_{t}} (\pa_\sigma R_t)^2   \, d \Ha^1\right)  = &-  \int_{\pa F_{t}}  g(\nu_t) (\pa_{\sigma \sigma \sigma}R_t)^2\, d \Ha^1  + \int_{\pa F_{t}}   \pa_{\sigma \sigma }R_t \frac{\pa}{\pa t} \left( Q(E(u(\cdot,t)) \right) \, d \Ha^1\\ 
&+   \int_{\pa F_{t}}   g(\nu_t) k_t^2  (\pa_{\sigma \sigma}R_t)^2 \, d \Ha^1  +  \int_{\pa F_{t}}     \pa_{\nu_t} (Q(E(u_t)))  (\pa_{\sigma \sigma}R_t)^2 \, d \Ha^1 \\
  &-  \frac{1}{2}   \int_{\pa F_{t}}   k_t \,(\pa_{\sigma}R_t)^2 \pa_{\sigma \sigma}R_t \, d \Ha^1.
\end{split}
\eeq
In order to conclude, we need  to show  that 
$$
\int_{\pa F_{t}}   \pa_{\sigma \sigma }R_t  \frac{\pa}{\pa t} \left( Q(E(u(\cdot,t)) \right) \, d \Ha^1 = 2 \int_{\Omega \setminus  F_{t}} Q(E(\dot u_t)) \, dx
$$
to recognize the quadratic form $ - \pa^2 J[\pa_{\sigma\sigma}R_t]$ in the four first terms of \eqref{app7}. 
To this aim, observe that since $\C E(u_t)[\nu_t] = 0$ on $\pa F_t$, we have
\[
\begin{split}
\int_{\pa F_{t}}   \pa_{\sigma \sigma }R_t &\frac{\pa}{\pa t} \left( Q(E(u(\cdot,t)) \right) \, d \Ha^1 = \int_{\pa F_{t}}   R_{\sigma \sigma }  \C E(u_t) : E(\dot u_t) \, d \Ha^1 =   \int_{\pa F_{t}}   R_{\sigma \sigma }  \C E(u_t) : D \dot u_t \, d \Ha^1 \\
&= \int_{\pa F_{t}}   \pa_{\sigma \sigma }R_t  \C E(u_t) : D_\tau \dot u_t \, d \Ha^1 = - \int_{\pa F_{t}} \Div_\tau (\pa_{\sigma \sigma }R_t  \C E(u_t) ) \cdot \dot u_t   \, d \Ha^1\\
&=2 \int_{\Omega \setminus  F_{t}} Q(E(\dot u_t)) \, dx, 
\end{split}
\]
where the last equality follows by choosing $\varphi = \dot u_t$ as a test function in the equation~\eqref{eq u dot}.~\end{proof}

\begin{proof}[\textbf{Proof of Theorem \ref{thmstability}.}]
 Throughout the proof,   $C$ will denote a constant depending only on the $C^{2,\alpha}$-bounds on the boundary of the set $G$.  Here we always assume that $\alpha < 1/2$ and the value of $C$  may change from line to line.   
For any set $F \in \mathfrak{h}^{2,\alpha}_M(G)$  consider  
 \beq\label{D(F)0}
 D(F):=\int_{F\Delta G}\mathrm{dist\,}(x, \pa G)\, dx
 \eeq
 and note that 
 $$
 |F\Delta G|\leq C\|h_F\|_{L^2(\pa G)}\leq  C \sqrt{D(F)}\,
 $$
for constants depending only on $G$. Recall that  $h_F$ is the function such that 
$$
\pa F = \{ x +h_F(x)\nu_G(x) : x \in \pa G \}.
$$

For every  $\e_1>0$ sufficiently small, there exists $\de_1\in (0,1)$ so small that for any set $F\in \mathfrak{h}^{2,\alpha}_M(G)$ the following implications hold true:
 \beq\label{de01}
 F\in \mathfrak{h}^{2,\alpha}_M(G)\text{ and } D(F)\leq \de_1\Longrightarrow \|h_F\|_{C^{1,\alpha}(\pa G)}\leq \frac{\e_1}2\,,
 \eeq  
 and  
  \beq\label{de02}
  \|h_F\|_{C^{1,\alpha}(\pa G)}\leq \e_1\text{ and } \int_{\pa F} (\pa_\sigma R_F)^2\, d \Ha^1 \leq 1\Longrightarrow \|h_F\|_{C^{2,\alpha}(\pa G)}\leq \omega(\e_1)\leq\sigma_1\land M,
 \eeq  
where $\omega$ is a positive non-decreasing function such that $\omega(\e_1)\to 0$ as $\e_1\to 0^+$, and $\sigma_1$ is the constant provided by Proposition~\ref{stationary}. Here $R_F$ stands, as usual,  for $g(\nu_F)k_F-Q(E(u_F))$ on $\pa F$. 

 Fix $\e_1$, $\de_1\in (0,1)$ satisfying \eqref{de01} and \eqref{de02} and choose an initial set $F_0\in \mathfrak{h}^{2,\alpha}_M(G)$ such that
 \beq\label{initial}
 D(F_0)\leq \delta_0\qquad\text{and}\qquad \int_{\pa F_0} (\pa_\sigma R_0)^2\, d \Ha^1 \, dx \leq \delta_0\,,
 \eeq
where the choice of  $\delta_0 < \delta_1$ will be made later. Here, we denote $R_0$ instead of $R_{F_0}$. 

 Let $(F_t)_{t\in (0, T(F_0))}$ the unique classical solution of the  flow \eqref{flow} provided 
 by Theorem~\ref{th:existence}. Here $T(F)\in (0,+\infty]$ stands for  the maximal time of existence of the classical  solution starting from $F$.  By the same theorem, there exists $\de>0$ and $T_0>0$ such that 
 \beq\label{dalbasso}
 T(F)\geq T_0\quad\text{for  all $F\subset\subset\Om$  s.t. $\|h_F\|_{L^2(\pa G)}\leq \de$ and $\|h_F\|_{H^{3}(\pa G)}\leq 1$.}
 \eeq
Without loss of generality, in what follows we may also assume $\de_1$ to be so small that   
$ D(F)\leq \de_1$ implies   $\|h_F\|_{L^2(\pa G)}\leq \de$, with $\de$ as in \eqref{dalbasso}.

We now split the rest of the proof into two steps. 

\medskip 

\noindent {\bf Step 1.}{\it (Stopping-time)} Let $\bar t\leq T(F_0)$ be the maximal time such that 
\beq\label{Tprimo}
\|h_t\|_{C^{1,\alpha}(\pa G)}<\eps_1\quad\text{and}\quad \int_{\pa F_t}  (\pa_\sigma R_t)^2\, d \Ha^1< 2\delta_0.
\quad\text{for all $t\in (0, \bar t)$,}
\eeq
 Note that such a maximal time is well defined in view of \eqref{de01} and \eqref{initial}. We claim that  by taking   $\eps_1$ and $\de_0$ smaller if needed, we have $\bar t=T(F_0)$.
To this aim, assume by contradiction that $\bar t< T(F_0)$. Then,   
\[
\|h_{\bar t}\|_{C^{1,\alpha}(\pa G)}= \eps_1 \quad\text{or}\quad \int_{\pa F_{\bar t}} (\pa_\sigma R_{\bar t})^2 \, d \Ha^1 =2\delta_0
\]
We  split  the proof into steps, according to the two alternatives above.

\vspace{5pt}
{\bf Step 1-(a).}  Assume that 
\beq\label{step1a}
\int_{\pa F_{\bar t}} (\pa_\sigma R_{\bar t})^2 \, d \Ha^1 =2\delta_0.
\eeq

Since  \eqref{de02} holds for $h_t$ for every $t \in (0,\bar t)$ then by Lemma ~\ref{lemma:j2>0near} (and the fact that $\sigma_1<\sigma_0$) we get  
\[
J(F_t)[\pa_{\sigma \sigma} R_t] \geq \frac{m_0}{2} \|\pa_{\sigma \sigma} R_t\|_{H^1(\pa F_t)}^2. 
\]
Therefore by Proposition \ref{energia identiteetit} we get
\beq \label{from_stability}
\begin{split}
\frac{d}{dt} \biggl(\frac{1}{2} \int_{\pa F_t} (\pa_\sigma R_t)^2& \, d\Ha^1 \biggr) = -J(F_t)[\pa_{\sigma \sigma} R_t] - \frac12 \int_{\pa F_t} k_t \, (\pa_\sigma R_t)^2 \, \pa_{\sigma \sigma}R_t\, d \Ha^1\\
&\leq  -\frac{m_0}{2} \|\pa_{\sigma \sigma} R_t\|_{H^1(\pa F_t)}^2  +C \int_{\pa F_t}  (\pa_\sigma R_t)^2 \, |\pa_{\sigma \sigma}R_t| \, d \Ha^1\\
&\leq  -\frac{m_0}{2} \int_{\pa F_t} (\pa_{\sigma \sigma\sigma} R_t)^2 d \Ha^1 + C \int_{\pa F_t}   |\pa_\sigma R_t|^3  +  |\pa_{\sigma \sigma}R_t|^3 \, d \Ha^1.
\end{split}
\eeq
In turn, by Proposition \ref{interpolation} and the Poincar\'e Inequality we get
\[
\begin{split}
\|\pa_\sigma R_t \|_{L^3(\pa F_t)}^3 &\leq C \, \|\pa_{\sigma \sigma \sigma}R_t \|_{L^2(\pa F_t)}^{\frac14}\|\pa_{\sigma}R_t \|_{L^2(\pa F_t)}^{\frac{11}{4}} \\
 &\leq C \|\pa_{\sigma \sigma \sigma}R_t \|_{L^2(\pa F_t)}^{2} \|\pa_{\sigma}R_t \|_{L^2(\pa F_t)}\\
&\leq C \sqrt{\delta_0}  \|\pa_{\sigma \sigma \sigma}R_t \|_{L^2(\pa F_t)}^{2},
\end{split}
\]
where the last inequality follows from \eqref{Tprimo}. Similarly we get
\[
\begin{split}
\|\pa_{\sigma \sigma}R_t\|_{L^3(\pa F_t)}^3 &\leq C \, \|\pa_{\sigma \sigma\sigma}R_t \|_{L^2(\pa F_t)}^{\frac74}\|\pa_{\sigma }R_t \|_{L^2(\pa F_t)}^{\frac{5}{4}} \\
&\leq C \sqrt{\delta_0}  \|\pa_{\sigma \sigma\sigma}R_t \|_{L^2(\pa F_t)}^{2}.
\end{split}
\]
Therefore, choosing $\delta_0$   small enough,  we deduce from \eqref{from_stability} that   
\[
\begin{split}
\frac{d}{dt} \biggl(\frac{1}{2} \int_{\pa F_t} (\pa_\sigma R_t)^2 \, d\Ha^1 \biggr) &\leq  \bigl( -\frac{m_0}{2} +  C \sqrt{\delta_0}\bigr) \int_{\pa F_t} (\pa_{\sigma \sigma\sigma} R_t)^2 \, d \Ha^1\\
&\leq -\frac{m_0}{4}  \int_{\pa F_t} (\pa_{\sigma \sigma\sigma} R_t)^2 \, d \Ha^1 \\
&\leq -m_1   \int_{\pa F_t} (\pa_{\sigma} R_t)^2 \, d \Ha^1
\end{split}
\]
for all $t \leq \bar t$ and for some $m_1>0$. Note that the last inequality above follows from the Poincar\'e Inequality. Integrating the above differential inequality   implies
\beq \label{exponential}
 \int_{\pa F_t} (\pa_\sigma R_t)^2 \, d\Ha^1 \leq e^{-m_1 t}  \int_{\pa F_0} (\pa_\sigma R_0)^2 \, d\Ha^1 \leq \delta_0\, e^{-m_1 t} 
\eeq
for every $t \leq \bar t$.  This contradicts \eqref{step1a}. 

\vspace{5pt}
{\bf Step 1-(b).}  Assume now that 
\beq\label{step1b}
\|h_{\bar t}\|_{C^{1,\alpha}(\pa F)}=\e_1\,.
\eeq

By  \eqref{stimaf3} we have that 
\beq\label{leading2}
\frac{d}{dt} D(F_t) \leq P(F_t)^{\frac12} \left( \int_{\pa F_t} (\pa_\sigma R_t)^2 \, d\Ha^1 \right)^{\frac12},
\eeq
where $D(F_t)$ is defined in \eqref{D(F)0}. Therefore we may use \eqref{exponential} to estimate 
\beq \label{L2 decay}
\frac{d}{dt} D(F_t) \leq C  \sqrt{\delta_0 e^{-m_1  t}} \, 
\eeq
for every $t \leq \bar t$. This implies 
\[
D(F_t)  \leq D(F_0) +  C\sqrt{\delta_0} \leq  C\sqrt{\delta_0}
\]
for  every $t \leq \bar t$.  We may choose  $\delta_0$ so small  enough  the above estimate implies  $D(F_t) \leq \delta_1$ and, in turn, by   
 \eqref{de01},  $\|h_{t}\|_{C^{1,\alpha}(\pa F)}\leq \frac{\eps_1}{2}\,$ for every $t \leq \bar t$. This contradicts \eqref{step1b}. 

\medskip

\noindent {\bf Step 2.}{\it (Global-in-time  existence and convergence)} By the previous step we have that as long as the flow is defined, i.e., over $(0,T(F_0))$, the estimates \eqref{Tprimo} hold.  In turn,  by taking $\e_1$ (and $\de_1$, $\de_0$) smaller if needed,  we have    $\|h_t\|_{L^2(\pa G)} \leq \de$ and $\|h_t\|_{H^3(\pa G)} \leq 1$   for all  $t\in (0,T(F_0))$. By \eqref{dalbasso} and a standard continuation argument,  we deduce that  $(F_t)_{t}$ is defined for all times, i.e.,  $T(F_0) = \infty$. 

From \eqref{exponential} we also deduce that 
\beq \label{exponential2}
\int_{\pa F_t} (\pa_\sigma R_t)^2 \, d\Ha^1 \leq e^{-m_1 t}  \int_{\pa F_0} (\pa_\sigma R_0)^2 \, d\Ha^1 \leq \delta_0\, e^{-m_1 t} 
\eeq
 for all $t >0$, and in turn,  by  \eqref{de02} we have 
\beq \label{exponential3}
\|h_t\|_{C^{2,\alpha}(\pa G)}\leq M 
\eeq
for all $t>0$.  Therefore, we deduce that there exists $h_\infty \in C^{2,\alpha}(\pa G)$ and a sequence $t_n\to+\infty$ such that 
\beq\label{tenne}
h_{t_n} \to h_\infty\qquad\text{ in $C^{2,\beta}(\pa G)$ for all $\beta < \alpha$.}
\eeq
 Moreover, by \eqref{exponential2} we have $\pa_\sigma R_{\infty}=0$, and thus, the set $F_\infty\in \mathfrak{h}^{2,\alpha}_M(G)$ such that
\[
\pa F_\infty = \{ x + h_\infty(x) \nu_G(x) : x \in \pa G\}
\]  
is stationary.  Recall that for every $t\in [0, +\infty]$, $(\Gamma_{F_t,i})_{i=1,\dots, m}$ denote the connected components of $\pa F_t$, 
numbered according to \eqref{numbered}. Denote also as usual by $F_{t,i}$ the bounded open set enclosed by $\Gamma_{F_t,i}$.
Since $|F_{t,i}|=|F_{i,0}|$ for every $t>0$ by \eqref{componenti}, taking also into account \eqref{de02} and Proposition~\ref{stationary}, we deduce that in fact $F_\infty$ is the unique stationary set in $\mathfrak{h}^{2,\alpha}_{\sigma_1}(\pa G)$
 such that $|F_{\infty, i}|=|F_{0,i}|$ for $i=1,\dots, m$.

It remains to show that the whole flow exponentially converges to $F_\infty$.
To this aim, define 
$$
D_\infty(E):=\int_{E\Delta F_\infty}\mathrm{dist\,}(x, \pa F_\infty)\, dx\,.
$$
The  same calculations and arguments leading to \eqref{leading2}  and \eqref{L2 decay}  show that 
\beq\label{step6}
\frac{d}{dt} D_\infty(F_t) \leq P(F_t)^{\frac12} \left( \int_{\pa F_t} (\pa_\sigma R_t)^2 \, d\Ha^1 \right)^{\frac12}\leq C  \sqrt{\delta_0 e^{-m_1  t}}
\eeq
for all $t>0$. From this inequality it is easy to deduce that $\lim_{t\to +\infty} D_\infty(F_t)$ exists. Thus, by \eqref{tenne},  
$D_\infty(F_t)\to 0$ as $t\to +\infty$. In turn, integrating \eqref{step6} and writing $\pa F_t=\{x+\tilde h_t(x)\nu_{F_\infty}(x): x\in \pa F_\infty\}$ we get
$$
\|\tilde h_t\|_{L^2(\pa F_\infty)}^2\leq C D_\infty(F_t)\leq\int_t^{+\infty}C  \sqrt{\delta_0 e^{-m_1  s}}, ds\leq 
C  \sqrt{\delta_0 e^{-m_1  t}}\,.
$$
Since $(\tilde h_t)_{t>0}$ are  bounded in $C^{2,\alpha}(\pa G)$  by \eqref{exponential3}, we obtain by the above estimate together with standard interpolation that also $\|\tilde h_t\|_{C^{2,\beta}(\pa G)}\to 0$  exponentially fast to zero for $\beta < \alpha$. Finally, using also \eqref{exponential2} and Lemma~\ref{elastiset} (with $G=F_\infty$), we deduce  that $\|\tilde h_t\|_{H^3(\pa G)} \to 0$ exponentially fast.  
\end{proof}

\section{Periodic graphs}\label{sec:graphs}

In this section we briefly describe how our main results read in the context of evolving periodic graphs.

In this framework, given a (sufficiently regular) non-negative $\ell$-periodic function $h:[0,\ell]\to [0,+\infty)$, the free energy associated with it  reads
\begin{equation}\label{functional}
J(h):=\int_{\Omega_{h}}Q(E(u_h))\,dx+\int_{\Gamma_{h}}\vphi(\nu_{\Om_h})\,d{\mathcal{H}}^{1},%
\end{equation}
where $x=(x_1,x_2)\in\mathbb{R}^{2}$,  $\Gamma_{h}$ denotes the graph of $h$, $\Omega_{h}$ is the subgraph of $h$, i.e.,
\[
\Omega_{h}:=\{(x_1,x_2)\in(0,\ell  )\times\mathbb{R}:\,0<x_2<h(x_1)\},
\]
and $u_h$ is the elastic equilibrium in $\Om_h$, namely the solution of  the elliptic system
\beq\label{leiintro}
\begin{cases}
	\Div\,\C E(u_h)=0 \quad \text{in $\Om_{h}$},\\
	\C E(u_h)[\nu_{\Om_h}]=0 \quad \text{on $\Gamma_h$,}\\
	\nabla u_h(\cdot, x_2)  \quad \text{ is $\ell$-periodic,}\\
	 u(x_1,0)=e_0(x_1,0)\,,
\end{cases}
\eeq
for a suitable fixed constant $e_0\neq 0$.  As mentioned already in the introduction, the above energy relates to a variational model for epitaxial growth: The graph $\Gamma_h$ describes the (free) profile of the elastic films, which occupies the region $\Om_h$ and is grown on a (rigid) and much thicker substrate, while the \emph{mismatch strain} constant $e_0$ appearing in the Dirichlet condition for $u_h$ at the interface $\{x_1=0\}$ between film and substrate measures the mismatch between the characteristic atomic distances in the lattices of the two materials.   
  In this framework, the (local) minimizers of \eqref{functional} under an area constraint on $\Om_h$ describe the equilibrium configurations of epitaxially strained elastic films, see \cite{FFLM, FFLM2, FFLM3, FM09} and the reference therein. 

In the context of periodic graphs, given an initial $\ell$-periodic profile $\bar h\in H^3(0,\ell)$ (in short $\bar h\in H^3_{\rm per}(0,\ell)$), we look for a solution $(h_t)_{t\in [0,T)}$ of the following problem:
\beq\label{flow2}
\begin{cases}
	\frac 1{J_t}\dot h_t=\left(g(\nu_t)k_t+Q(E(u_t))\right)_{\sigma\sigma}
	& \text{on $\Gamma_{h_t}$ and for all $t\in(0, T)$,}\\
		h_t \text{ is $\ell$-periodic}& \text{for all $t\in(0, T)$,}\\
	h_0=\bar h\,,
\end{cases}
\eeq
where we set 
$J_t:=\sqrt{1+\left|\tfrac{\partial h_t}{\partial x_1} \right|^2},
$  
 $u_t$ stands for the solution of  \eqref{leiintro}, with $\Om_{h_t}$ in place of $\Om_h$, and we wrote $\nu_t$, $k_t$ instead of $\nu_{\Om_{h_t}}$ and $k_{\Om_{h_t}}$, respectively.  Note that in the first equation of \eqref{flow2} we have $+Q(E(u_t))$ instead of $-Q(E(u_t))$. This is due to the fact that in 
 \eqref{functional}  the vector $\nu_{\Om_h}$ now point outwards with respect to the elastic body.

Although the setting is slightly different from that of the previous sections,  the short-time existence  and regularity theory of Section~\ref{sec:existence} clearly extends also to the present situation, with the same arguments. In this way we improve upon the results of \cite{FFLM2}, showing that there is no need of a curvature regularization in the case where the anisotropy 
$\vphi$ is convex and satisfies the ellipticity condition \eqref{ellipticity}. Also the stabilty analysis of Section~\ref{sec:stability} goes through without changes, thus showing that strictly stable stationary $\ell$-periodic configurations are  $H^3$-exponentially stable (in the sense made precise by Remark~\ref{rm:precise}). 

In the case of flat configurations, that is, of constant profiles $h\equiv a$ for some $a>0$, and when $Q$ is of the form
$$
Q(E):=\mu|E|^2+\frac{\lambda}2(\mathrm{trace\,} E)^2
$$
for some constants $\mu>0$ and $\lambda>-\mu$ (the so called \emph{Lam\'e coefficients}), the  relation between the $a$, $\mu$, $\lambda$, $\ell$, and $e_0$ (see \eqref{leiintro}) that guarantees the strict stability of flat configuration $h\equiv a$ with respect to $\ell$-periodic perturbations is analytically determined.  For the reader's convenience, we recall the results. Consider the  
 {\em Grinfeld function} $K$ defined by
$$
K(s):=\max_{n\in\N}\frac{1}{n}H(ns)\,, \quad s\geq0\,,
$$
where
$$
H(s):=\frac{y+(3-4\nu_p)\sinh s\cosh s}{4(1-\nu_p)^2+s^2+(3-4\nu_p)\sinh^2s}\,,
$$
and $\nu_p$ is the {\em Poisson modulus} of the elastic material, i.e.,  
$\nu_p: =\frac{\lambda}{2(\lambda+\mu)}$\,.

It turns out that $K$ is strictly increasing and continuous,  $K(s)\leq Cs$, and  $\displaystyle\lim_{s\to +\infty}K(s)=1$, for some positive constant $C$, see \cite[Corollary~5.3]{FM09}.
Set $\mathbf{e_1}:=(1,0)$ and  $\mathbf{e_2}:=(0,1)$.
 Combining \cite[Theorem~2.9]{FM09} and \cite[Theorem~2.8]{Bo0} with the results of the previous section, we obtain the following theorem.
\begin{theorem}\label{th:2dliapunov}
  Let $\dloc:(0,+\infty)\to (0,+\infty]$ be defined as $\dloc (\ell):=+\infty$,  if $0<\ell\leq \frac{\pi}{4}\frac{(2\mu+\lambda)\partial_{\mathbf{e_1}\mathbf{e_1}}\vphi(\mathbf{e_2})}{e_0^2\mu(\mu+\lambda)}$, and  as the solution of  
$$
K\Bigl(\frac{2\pi \dloc(\ell)}{\ell}\Bigr)=\frac{\pi}{4}\frac{(2\mu+\lambda)\partial_{\mathbf{e_1}\mathbf{e_1}}\vphi(\mathbf{e_2})}{e_0^2\mu(\mu+\lambda)}\frac1\ell\,,
$$
otherwise.
Then  $h\equiv a$ is an $\ell$-periodic strictly stable stationary configuration for \eqref{functional}  if and only if 
 $0<a<\dloc(\ell)$.
In particular, for all  $a\in (0,\dloc(\ell))$  there exists $\de>0$ such that if $\|\bar h-a\|_{H^3_{\rm per}(0,\ell)}\leq \de$ and 
$|\Om_{\bar h}|=a\ell$, then the unique solution $(h_t)_t$ of \eqref{flow2} is defined for all times and satisfies
$$
\|h_t-a\|_{H^3_{\rm per}(0,\ell)}\leq C e^{-c t} \qquad\text{for all $t>0$,}
$$
for suitable constants $C, c>0$.
\end{theorem}

\section{Appendix}

Let $s\in (0,1)$ and $p\geq 1$. We recall that for a function $f: \mathbb{S}^1\to \R$ the Gagliardo seminorm $[f]_{s,p}$ is defined as
$$
[f]^p_{s,p}:=\int_{\mathbb{S}^1}\int_{\mathbb{S}^1}\frac{|f(x)-f(y)|^p}{|x-y|^{1+sp}}\, dxdy\,.
$$
If $s>0$ and $\ell$ is the integer part of $s$, the Sobolev space  $W^{s,p}(\mathbb{S}^1)$ is the space of all functions $f$ in 
$W^{\ell, p}(\mathbb{S}^1)$ such that  $[\partial^\ell f]_{s-\ell,p}$ is finite, endowed with the norm 
$\|f\|^p_{W^{s,p}(\mathbb{S}^1)}:=\|f\|_{W^{\ell, p}(\mathbb{S}^1)}^p+[\partial_\sigma^{\ell}f]_{s-\ell,p}^p$. Here we used the convention $W^{0,p}=L^p$ and $[\partial_\sigma^{\ell}f]_{s-\ell,p}^p=\|\partial_\sigma^{\ell}f\|_{L^p(\Gamma)}$. We recall also that for 
$p=2$ the  seminorm $[\partial_\sigma^{\ell}f]_{s-\ell,p}^p$ is equivalent to 
$$
\bigg(\sum_{k\in \Z}  k^{2s} a_k(f)^2  \bigg)^{\frac12}\,,
$$
where $\{a_k(f)\}$ is the sequence of the Fourier coefficients of $f$ with respect to the orthornormal basis $\{(2\pi)^{\frac12}\mathrm{e}^{-ikz}\}_{k\in \Z}$. These definitions extend in the obvious way to the case where $\mathbb{S}^1$ is replaced by any regular Jordan curve $\Gamma$.

We prove the following interpolation inequality for   curves. Note that in the statement we are using  and $W^{t,2}=H^t$ for all $t>0$.
\begin{proposition} 
\label{interpolation}
Let $\Gamma$ be a regular Jordan curve. Let $m\geq 1$ be an integer, $0\leq s< m$ and 
$p\in [2,+\infty)$ such that $s+ 1/2 -1/p<m$.  There exists a  constant $C>0$,   depending only on $m, s, p$ and on the length of $\Gamma$  such that for every $f\in H^{m}(\Gamma)$
\begin{equation}\label{inter1}
\|f\|_{W^{s,p}(\Gamma)} \leq C\big(\|\partial_\sigma^m f\|_{L^2(\Gamma)}^\theta \|f\|_{L^2(\Gamma)}^{1-\theta} + \|f\|_{L^2(\Gamma)} \big), 
\end{equation}
where
\[
\theta = \frac{s+ 1/2 -1/p}{m}.
\]
If $s$ is a positive  integer, then 
\begin{equation}\label{inter2}
|| \partial_\sigma^s f||_{L^p(\Gamma)} \leq C\,|| \partial_\sigma^m f||_{L^2(\Gamma)}^\theta ||f||_{L^2(\Gamma)}^{1-\theta}\,,
\end{equation}
with $\theta$ as before. The same inequality also holds if $s=0$, provided that $f$ has zero average. 

Finally, if $0<\alpha<\frac12$, there  exists $\theta'$, depending only on $m$ and $\alpha$, such that for every $f\in H^{m}(\Gamma)$
\begin{equation}\label{inter3}
\|f\|_{C^{m-1, \alpha}(\Gamma)}\leq C (\|\partial^m_\sigma f\|_{L^2(\Gamma)}^{\theta'}\|f\|_{L^2(\Gamma)}^{1-\theta'}+\|f\|_{L^2(\Gamma)})\,.
\end{equation}
\end{proposition}
\begin{proof}
It is enough to prove the statement for $\Gamma=\mathbb{S}^1$. The general case will follow by parametrizing $\Gamma$ by the arclength and by rescaling.   Let $t=s+\frac12-\frac1p$.
Observe that 
\beq\label{inter4}
\bigg(\sum_{k\in \Z}^\infty k^{2t}a_k(f)^2 \bigg)^{\frac12}\leq 
\bigg(\sum_{k\in \Z} k^{2m} a_k(f)^2 \bigg)^{\frac\theta2}
\bigg(\sum_{k\in \Z} a_k(f)^2 \bigg)^{\frac{1-\theta}2}\,,
\eeq
and thus \eqref{inter1} follows with $\|f \|_{W^{s,p}(\Sb^1)}$ replaced by 
$\|f \|_{W^{t,2}(\Sb^1)}$. The general case follows recalling that 
$W^{t,2}(\Sb^1)$ is continuously embedded in $W^{s,p}(\Sb^1)$, since $t=s+\frac12-\frac1p$ 
(see \cite[Th. 1.4.4.1]{Gr}). 
Observe that it is enough to prove \eqref{inter2} for functions $f$ with zero average also when 
 $s$ is a positive integer. On the other hand if $f$ has zero average, \eqref{inter2} follows from 
 \eqref{inter4} and the aforementioned Sobolev Embedding after observing that the 
 $W^{s,p}$-norm of $f$ is equivalent to the $L^p$-norm of $\partial_\sigma^s f$.
 
 Finally, to prove \eqref{inter3} it is enough to assume $m=1$ and then to argue by induction with respect to $m$. To this aim, we observe that for every $z$, $w\in\Sb^1$
 $$
 |f(z)-f(w)|\leq c |z-w|^{\frac12}\|\partial_\sigma f\|_{L^2}\,,
 $$ 
for some universal constant $c$. Then if $0<\alpha<\frac12$
\begin{eqnarray*}
|f(z)-f(w)|\leq |f(z)-f(w)|^{2\alpha}|f(z)-f(w)|^{1-2\alpha}\leq c |z-w|^{\alpha} \|\partial_\sigma f\|_{L^2}^{2\alpha}\|f\|_{L^\infty}^{1-2\alpha}\,.
\end{eqnarray*} 
The conclusion follows by estimating $\|f\|_{L^\infty}$ by \eqref{inter1}, with $m=1$.
\end{proof}

\end{document}